\documentclass[a4paper,12pt]{amsart}

\usepackage{mathrsfs, amssymb, array, wasysym}
\usepackage[all,matrix,arrow,curve]{xy}
\usepackage[all]{xy}
\usepackage{mathcomp}
\usepackage{verbatim} 

\makeatletter 
\@addtoreset{equation}{subsection} 
\makeatother

\newcommand{\Cl}{\operatorname{Cl}}

\newcommand{\Supp}{\operatorname{Supp}}
\newcommand{\Sing}{\operatorname{Sing}}
\newcommand{\red}{\operatorname{red}}

\newcommand{\wt}{\operatorname{wt}}
\newcommand{\ord}{\operatorname{ord}}
\newcommand{\rk}{\operatorname{rk}}

\newcommand{\gr}{\operatorname{gr}}
\newcommand{\cd}{\operatorname{cdisc}}
\newcommand{\dd}{\operatorname{d}}
\newcommand{\F}{\operatorname{F}}
\newcommand{\unit}{\operatorname{unit}}

\newcommand{\toplus}{\mathbin{\tilde\oplus}}
\newcommand{\totimes}{\mathbin{\tilde\otimes}}

\newcommand{\CC}{\mathbb{C}}
\newcommand{\ZZ}{\mathbb{Z}}
\newcommand{\PP}{\mathbb{P}}
\newcommand{\QQ}{\mathbb{Q}}

\newcommand{\lin}{\text{---}}

\newcommand{\KKK}{{\mathscr{K}}}

\newcommand{\OOO}{\mathscr{O}}

\newcommand{\muu}{{\boldsymbol{\mu}}}

\newcommand{\type}[1]{$\mathrm{#1}$}

\newtheorem{thm}[subsection]{}

\newtheorem{sthm}[equation]{}
\newtheorem{ssthm}{}[equation]

\newtheorem*{lemma*}{Lemma}

\theoremstyle{definition}
\newtheorem{de}[subsection]{}
\newtheorem{sde}[equation]{}
\newtheorem{ssde}[ssthm]{}
\newtheorem*{computation*}{Computation}
\newtheorem*{remark*}{Remark}

\newcounter{eqnumerc}[equation]

\newcommand{\xref}[1]{{\rm \ref{#1}}}

\title{Threefold extremal contractions \\ of types \type{(IC)} and \type{(IIB)}}

\author{Shigefumi Mori}

\address{Shigefumi Mori: RIMS, 
Kyoto University, Oiwake-cho, Kitashirakawa, Sakyo-ku, Kyoto
606-8502, Japan}
\email{mori@kurims.kyoto-u.ac.jp}

\author{Yuri Prokhorov}
\thanks{
The first author's work  partially supported by Japan Society for the Promotion of Science Grant-in-Aid for
Scientific Research (B)(2), No. 20340005.
\newline\indent
The second author's work partially supported by 
RFBR grants No. 11-01-00336-a, 11-01-92613-KO\_a, the grant of
Leading Scientific Schools No. 4713.2010.1 and 
AG Laboratory SU-HSE, RF government 
grant ag. 11.G34.31.0023.
}

\address{Yuri Prokhorov: Department 
of Algebra, Faculty of Mathematics, Moscow State
University, Moscow 117234, Russia
\newline\indent
Laboratory of Algebraic Geometry, SU-HSE, 
7 Vavilova Str., Moscow 117312, Russia
}
\email{prokhoro@gmail.com}

\begin{document}

\maketitle
\begin{flushright}
\end{flushright}

\begin{abstract} 
Let $(X,C)$ be a germ of a threefold $X$ with terminal singularities
along an irreducible reduced complete curve $C$
with a contraction $f: (X,C)\to (Z,o)$ 
such that $C=f^{-1}(o)_{\red}$ and $-K_X$
is ample. Assume that $(X,C)$ contains a point 
of type \type{(IC)} or \type{(IIB)}. 
We complete the classification of such germs
in terms of a general member 
$H\in |\OOO_X|$ containing $C$. 
\end{abstract}

\section{Introduction}
\begin{de} 
\label{setup}
Let $(X,C)$ be a germ of a threefold with terminal singularities
along an reduced complete curve. We say that $(X,C)$
is an \textit{extremal curve germ} if 
there is a contraction $f: (X,C)\to (Z,o)$ such that
$C=f^{-1}(o)_{\red}$ and $-K_X$
is $f$-ample.  

If furthermore $f$ is birational, then $(X,C)$ is said to be 
an \textit{extremal neighborhood} \cite{Mori-1988}.
In this case $f$ is called \textit{flipping} if 
its exceptional locus coincides with $C$ 
(and then $(X,C)$ is called \textit{isolated}).
Otherwise the exceptional locus of $f$ is two-dimensional
and $f$ is called \textit{divisorial}.
If $f$ is not birational, then $\dim Z=2$ and 
$(X,C)$ is said to be 
a \textit{$\QQ$-conic bundle germ} \cite{Mori-Prokhorov-2008}.

\end{de}
\begin{de}
In this paper we consider only extremal curve germs with irreducible 
central fiber $C$.
For each singular point $P$ of $X$ with $P \in C$, consider
the germ $(P \in C' \subset X)$. All such germs are classified into  types 
\type{IA}, \type{IC}, \type{IIA}, \type{IIB}, 
\type{IA^\vee},  \type{II^\vee}, \type{ID^\vee}, \type{IE^\vee},
and \type{III}, 
whose definitions we 
refer the reader to \cite{Kollar-Mori-1992} and \cite{Mori-Prokhorov-2008}.

In this paper we complete the classification of  extremal curve germs 
with irreducible central fiber containing 
points of type   \type{IC} or \type{IIB}.
As in \cite{Kollar-Mori-1992} and \cite{Mori-Prokhorov-2010} 
the classification is done in terms of
a general hyperplane section, 
that is, a general divisor $H$ of $|\OOO_X|_C$, 
the linear subsystem of $|\OOO_X|$ consisting of sections  containing $C$.

For a normal surface $S$ and a curve $V\subset S$,
we use the usual notation of graphs $\Delta (S,V)$
of the minimal resolution of $S$ near $V$:
each $\diamond$ corresponds to an irreducible component of $V$ and each
$\circ$ corresponds to an exceptional divisor on the minimal
resolution of $S$, and we may use $\bullet$ instead of $\diamond$ if we want 
to emphasize that it is a complete $(-1)$-curve.
A number attached to a vertex denotes the minus self-intersection number.
For short, we may omit $2$ if the self-intersection is $-2$.

Recall that  if an  extremal curve germ $(X,C\simeq \PP^1)$ contains 
a point   of type   \type{IC}, then $(X,C)$ is not divisorial \cite[Cor. 8.3.3]{Kollar-Mori-1992}.
For the  remaining $\QQ$-conic bundle case we prove the following.
\end{de}

\begin{thm}\label{IC-main}{\bf Theorem.}
Let $(X,C)$ is a $\QQ$-conic bundle germ of type \type{(IC)}
with irreducible $C$ and let $f:(X,C)\to (Z,o)$ be the corresponding contraction.
Let $P\in X$ be \textup(a unique\textup) singular point.
Then we have:

\begin{sthm}
\label{IC-(8.3.2)}
The point $P\in X$ is of index $5$. 
Moreover,
the general member $H\in |\OOO_X|_C$ is normal, smooth outside of $P$,
has only rational singularities, and the following is the only possibility
for the dual graph of $(H,C)$\textup:
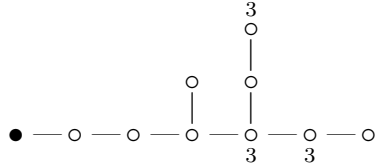
\begin{equation*} 
\begin{array}{c@\,c@\,c@\,c@\,c@\,c@\,c@\,c@\,c@\,c@\,c@\,c@\,c}
&&&&& &&& \overset3{\circ} 
\\[-4pt]
&&&&& &&&| 
\\ [-4pt]
&&&&& &\circ& &\circ 
\\ [-4pt]
&&&&& & | && | 
\\[-4pt] 
\bullet&\lin& \circ&\lin&\circ&\lin&\circ&\lin
&\underset3{\circ}&\lin&\underset3{\circ}&\lin&\circ 
\end{array}
\end{equation*} 
\end{sthm}
\end{thm}

If an extremal curve germ $(X,C\simeq \PP^1)$ contains a point of type 
\type{(IIB)}, then it cannot be flipping  \cite[Theorem 4.5]{Kollar-Mori-1992}.
Remaining cases of divisorial contractions and $\QQ$-conic bundles
are covered by the following theorem. 

\begin{thm}
\label{main-IIB}
{\bf Theorem.}
Let  $(X,C)$ is an extremal curve germ of type \type{(IIB)}
with irreducible $C$ and let $f:(X,C)\to (Z,o)$ be the corresponding contraction.
Let $P\in X$ be \textup(a unique\textup) singular point.
Then the general member $H\in |\OOO_X|_C$  is normal, smooth outside of $P$,
and has only rational singularities. Moreover, the following are the only possibilities
for the dual graph of $(H,C)$.
\par\medskip\noindent
{\bf $(X,P)$ is a simple \type{cAx/4} point (see \textup{\ref{IIB-local-definition}}):}
\begin{sthm}\label{IIB-theorem-A2case-simple}
$f$ is a divisorial contraction, $T:=f(H)$ is Du Val of type \type{A_2}\textup,
\[
\begin{array}{c@{}c@{}c@{}c@{}c@{}c@{}c@{}c@{}c}
\overset{3}\circ&\lin&\overset{4}{\circ}&\lin&\overset{}{\circ}&\lin&\circ&\lin&\circ
\\[-3pt]
&&|&&|
\\[-3pt]
&&\underset{3}{\circ}&&\circ&\lin&\bullet
\end{array}
\]
\end{sthm}

\begin{sthm}\label{IIB-theorem-smooth-case-simple}
$f$ is divisorial  contraction, $T:=f(H)$ is smooth\textup,
\[
\begin{array}{c@{}c@{}c@{}c@{}c@{}c@{}c@{}c@{}c@{}c@{}c@{}c@{}c@{}}
\overset{3}{\circ}&\lin&\circ&\lin&\circ&\lin&\circ&\lin&\circ&\lin&\circ&\lin&\bullet 
\\[-3pt]
&&&&&&|
\\[-3pt]
&&&&\underset3\circ&\lin&\underset4\circ
\end{array}
\]
\end{sthm}

\par\medskip\noindent
{\bf $(X,P)$ is a double \type{cAx/4} point:}

\begin{sthm}\label{IIB-theorem-D4case-double}
$f$ is divisorial  contraction, $T:=f(H)$ is Du Val of type \type{D_4}\textup,
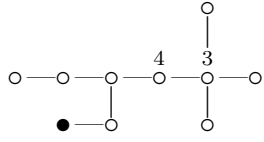
\[
\begin{array}{c@{}c@{}c@{}c@{}c@{}c@{}c@{}c@{} c@{} c@{}c}
&&&&&&&&\circ
\\[-5pt]
&&&&&&&&|
\\[-1pt]
\circ&\lin&\circ&\lin&\circ&\lin&\overset4\circ&\lin&\overset3\circ&\lin&\circ
\\[-5pt]
&&&&|&&&&|
\\[-5pt]
&&\bullet&\lin&\circ&&&&\underset{}{\circ}
\end{array}
\]
\end{sthm}

\begin{sthm}\label{IIB-theorem-conic-bundle-case-double}
$f$ is a $\QQ$-conic bundle\textup,
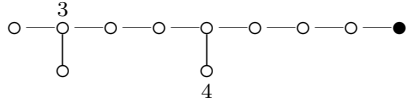
\[
\begin{array}{c@{}c@{}c@{}c@{}c@{}c@{}c@{}c@{}c@{}c@{}c@{}c@{}c@{}c@{}c@{}c@{}c@{}c@{}c@{}}
\circ&\lin&\overset{3}{\circ}&\lin&\overset{}\circ &\lin&\overset{}\circ&\lin&\overset{}{\circ}
&\lin&\overset{}\circ&\lin&\overset{}\circ&\lin&\overset{}{\circ}&\lin&\bullet
\\[-6pt]
&&|&&&&&&|
\\[-6pt]
&&\circ&&&&&&\underset{4}{\circ}
\end{array}
\]
\end{sthm}
\end{thm}

\section{Case \type{(IC)}}
In this section we prove Theorem \ref{IC-main}.
The techniques of \cite[ch. 8]{Kollar-Mori-1992} will be used freely, sometimes 
without additional explanations.

\begin{de}{\bf Setup.}
\label{IC-IC-notation-1}
Let $(X,P)$ be the germ of a three-dimensional terminal singularity and let $C\subset (X,C)$ be 
a smooth curve. 
Recall that the triple $(X,C,P)$ is said to be of type \type{(IC)} if 
there are analytic isomorphisms 
\begin{equation*} 
(X,P)\simeq\CC^3_{y_1,y_2,y_4}/\muu_m(2,m-2,1), \quad 
 C^\sharp\simeq\{y_1^{m-2}-y_2^2=y_4=0\}, 
\end{equation*}
where $m$ is odd and $m \ge 5$. 

\begin{sde}
Let $(X,C)$ be a $\QQ$-conic
bundle germ and let $f\colon (X, C) \rightarrow (Z, o)$ be 
the corresponding contraction. 
In this section we assume that $C$ is irreducible and has a point 
$P$ of type \type{(IC)}.
Recall that $(X,C)$ is locally primitive at $P$ \cite[4.2]{Mori-1988}.
Moreover, $P$ is the only singular point on $C$
\cite[Theorem 8.6, Lemma 7.1.2]{Mori-Prokhorov-2008}.
Thus the group $\Cl (Z,o)$ has no torsion.
Therefore, the base point $(Z,o)$ is smooth. 
\end{sde}
\end{de}

\begin{de}
\label{IC-(8.2)}
We have an $\ell$-splitting 
\begin{equation} \label{IC-eq--l-splitting-0}
\gr^1_C\OOO=(4P{^\sharp}) \toplus (-1 + (m-1)P{^\sharp}) 
\end{equation} 
by \cite[\S 3]{Mori-Prokhorov-2008III}, 
\cite[2.10.2]{Kollar-Mori-1992}, 
and hence the unique $(4P{^\sharp})$ in $\gr^1_C\OOO$. Since $y_4$ and 
$y^{m-2}_1-y^2_2$ form an $\ell$-free $\ell$-basis of 
$\gr^1_C\OOO$ at $P$, $(4P{^\sharp})$ has an $\ell$-free $\ell$-basis of the form
\begin{equation}
\label{IC-IC-notation-equation-u}
u=\lambda_1y^{(m-5)/2}_1y_4 + \mu_1(y^{m-2}_1-y^2_2)
\end{equation}
for some $\lambda_1$ and $\mu_1 \in \OOO_{C,P}$. It is easy to see that whether 
$\lambda_1(P) \neq 0$ does 
not depend on the choice of coordinates.

\begin{sde}{\bf Remark.}
We have 
\[
\OOO_C=\OOO_C(-H) \hookrightarrow \gr_C^1\OOO=\OOO\oplus \OOO(-1) 
\]
If $m\ge 7$, this implies that the term $y_1^2(y_1^{m-2}-y_2^2)$
appears in the equation of $H$.
If $m=5$, then either $y_1^2(y_1^3-y_2^2)$ or $y_1^2y_4$
appears in the equation of $H$.
\end{sde}
\end{de}

\begin{de}
\label{IC-ge}
According to \cite[\S 3]{Mori-Prokhorov-2008III} (cf. \cite[2.10]{Kollar-Mori-1992})
a general member $F\in |-K_X|$ contains $C$, has only Du Val singularities, and $\Delta(F,C)$ 
is the following graph of $(-2)$-curves
\begin{equation}
\label{IC-eq-Dn-diagram}
\begin{array}{c@\,c@\,c@\,c@\,c}
&&\bullet
\\[-4pt]
&&|
\\[-4pt]
\underbrace{\circ\lin\cdots\lin\circ}_{m-3}&\lin&\circ&\lin&\circ
\end{array} 
\end{equation}
where $\bullet$ corresponds to $C$.
We can choose coordinates $y_1$, $y_2$, $y_4$ in a neighborhood of $P$
so that $F=\{y_4=0\}/\muu_m$. 
In particular, the $\ell$-splitting \eqref{IC-eq--l-splitting-0} has the form
\begin{equation}
\label{IC-eq--l-splitting} 
\gr^1_C\OOO=(4P{^\sharp}) \toplus \OOO_C(-F). 
\end{equation} 
\end{de}

\begin{thm} {\bf Lemma.}
\label{lemma-IC-H}
A general member
$H\in |\OOO_X|_C$ is normal, has only rational singularities, and smooth outside of $P$.
\end{thm}
\begin{proof}
Similar to \ref{case-IIB-H-conic-bundle}.
Let $T:=f(H)$ and let $\Gamma:=H\cap F$. As in \ref{IIB-Stein-factorization}
consider the Stein factorization 
\begin{equation}
\label{IC=equation-Stein-factorization}
f_F: (F,C)\stackrel{f_1}\longrightarrow (F_Z, o_Z) \stackrel{f_2}\longrightarrow (Z,o).
\end{equation}
Put $\Gamma_Z:=f_1(\Gamma)$. 
 We may assume that, in some coordinate system, 
the germ $(F_Z, o_Z)$ is given by
$z^2+xy^2+x^{m-1}=0$. Then by \cite{Catanese-1987}
up to coordinate change the double cover $(F_Z, o_Z) \longrightarrow (Z,o)$
is just the projection to the $(x,y)$-plane.
Hence we may assume that $\Gamma_Z$ is given by $x=y$.
By \ref{IC-ge} we see that the graph $\Delta(F,\Gamma)$ has the form
\[
\begin{array}{c@\,c@\,c@\,c@\,c@\,c@\,c}
&\overset1\diamond&&&\overset1\bullet
\\[-4pt]
&|&&&|
\\[-4pt]
\underset1\circ\lin&\underset2\circ&\lin\cdots\lin\underset2\circ&\lin&\underset2\circ&\lin&\underset1\circ 
\end{array} 
\]
Therefore, $\Gamma$ is reduced and so $H$ is smooth outside of $P$.
The restriction $f_H: H\to T$ is a rational curve fibration.
Hence $H$ has only rational singularities.
\end{proof}

\begin{de}
\label{IC-(8.5)} 
Let $J$ be the $C$-laminal ideal such that 
$I_C \supset J \supset \F^2_C\OOO$ and 
$J/\F^2_C\OOO=(4P{^\sharp})$ in \eqref{IC-eq--l-splitting}. Since $J$ is locally a 
nested c.i. on $C\setminus\{P\}$ and $(y_4,u)$ is a (1,2)-monomializing 
$\ell$-basis of $I_C \supset J$ at $P$ with $u$ as in \eqref{IC-IC-notation-equation-u}. 
We have an $\ell$-exact sequence 
\begin{equation}
\label{IC-eq-(8.5.1)} 
0\to \OOO_C(-2F) \longrightarrow \gr^0_C J \longrightarrow (4P{^\sharp}) \to 0
\end{equation} 
and an $\ell$-isomorphism $\OOO_C(-2F) \simeq (-1+(m-2)P{^\sharp})$. Thus we have $\gr^0_C J \simeq \OOO 
\oplus \OOO(-1)$ as $\OOO_C$-modules. The unique $\OOO$ 
in $\gr^0_C J$ is generated near $P$ by 
\begin{equation}
\label{IC-eq-(8.5.2)} 
y^2_1u + \alpha y_2y^2_4 \mod \F^3(\OOO,J) 
\end{equation} 
for some $\alpha \in \OOO_{C,P}$. 
\end{de}

Proofs of the following two lemmas 
given in \cite{Kollar-Mori-1992} work in our situation 
without any changes.

\begin{thm}{\bf Lemma (\cite[Lemma 8.5.3]{Kollar-Mori-1992}).}
\label{IC-(8.5.3)}
\begin{equation*} 
\F^3(\OOO,J){^\sharp} \subset 
\left((y^{m-2}_1-y^2_2)^2,\, (y^{m-2}_1-y^2_2)y_4,\, \lambda_1y^{(m-5)/2}_1y^2_4,\, y^3_4\right). 
\end{equation*} 
\end{thm}

\begin{thm}{\bf Lemma (\cite[Lemma 8.6]{Kollar-Mori-1992}).}
\label{IC-lemma-l-split-(8.6)} 
The $\ell$-exact sequence \eqref{IC-eq-(8.5.1)} is 
$\ell$-split if and only if $\alpha (P)=0$. 
\end{thm}

\begin{thm} {\bf Proposition.}
\label{IC-Proposition-(8.7)} 
If $m \geq 7$, then $\alpha (P) \neq 0$. 
\end{thm} 

\begin{proof} Assume that $\alpha (P)=0$, that is, 
\eqref{IC-eq-(8.5.1)} is $\ell$-split. 
Then $\gr^0_C J$ contains a unique $(4P{^\sharp})$. 
Let $\KKK$ be the $C$-laminal 
ideal such that $J \supset \KKK \supset \F^1(\OOO,J)$ and $\KKK/\F^1(\OOO,J)=(4P{^\sharp})$. 
By \cite[8.14]{Mori-1988}, 
$\KKK$ is locally a nested c.i. on $C\setminus\{P\}$ and 
$(1,3)$-monomializable at $P$, and we have $\ell$-isomorphisms 
\begin{equation} 
\label{IC-eq-(8.7.1)} 
\gr^i_C(\OOO,\KKK) \simeq (-1+(m-i)P{^\sharp}),\qquad i=1,\, 2
\end{equation} 
and an $\ell$-exact sequence 
\begin{equation}
\label{IC-eq-(8.7.2)} 
0\to (-1+(m-3)P{^\sharp}) \longrightarrow \gr^3_C(\OOO,\KKK) 
\longrightarrow (4P{^\sharp}) \to 0. 
\end{equation} 
By \eqref{IC-eq-(8.7.1)}$\totimes\omega_X$, we see 
$\gr^i_C(\omega_X,\KKK)\simeq (-1+(m-i-1)P{^\sharp})$ and
so $H^j(\gr^i_C(\omega_X,\KKK))=0$ for $i=1,\, 2$, $j=0,\, 1$ because 
\[
m-2,\, m-3 \in 2\ZZ_{+} + (m-2)\ZZ_{+}. 
\]
Now using \eqref{IC-eq-(8.7.2)}
$\totimes\omega_X$, we obtain 
\begin{equation*} 
0\to (-2+(2m-4)P{^\sharp}) \longrightarrow 
\gr^3_C(\omega_X,\KKK) \longrightarrow (-1+(m+3)P{^\sharp}) \to 0. 
\end{equation*} 
We note $(-1+(m+3)P{^\sharp}) \simeq \OOO(-1)$ as 
$\OOO_C$-modules because $3 \notin 2\ZZ_{+}+(m-2)\ZZ_{+}$ for $m \geq 7$. We similarly note that 
$(-2+(2m-4)P{^\sharp}) \simeq \OOO(-2)$ because $m-4 \notin 2\ZZ_{+}+(m-2)\ZZ_{+}$. 
Hence, $H^1(\gr^3_C(\omega_X,\KKK)) \neq 0$. 
Note that $\omega_X/\F^1(\omega_X,\KKK)=\gr_C^0\omega\simeq \OOO(-1)$.
Using the standard exact sequences 
\begin{equation*} 
0\to \gr^i_C(\omega_X,\KKK)\longrightarrow \omega_X/\F^{i+1}(\omega_X,\KKK) 
\longrightarrow \omega_X/\F^i(\omega_X,\KKK) \to 0, 
\end{equation*} 
we obtain $H^1(\omega_X/\F^4(\omega_X,\KKK))\neq 0$.
By \cite[4.4]{Mori-Prokhorov-2008} we have
\[
-K_X\cdot C=5/m \ge -K_X\cdot f^{-1}(o)=2,
\] 
a contradiction.
\end{proof}

\begin{thm} {\bf Proposition.}
\label{IC-(8.8)} 
\begin{enumerate}
\item
\label{IC-(8.8.1)}
$\OOO_F(-C)$ is an $\ell$-invertible 
$\OOO_F$-module with an $\ell$-free $\ell$-basis $y^{m-2}_1-
y^2_2$ at $P$ and an $\ell$-isomorphism. 
\begin{equation*} 
\OOO_C\totimes\OOO_F(-C) \simeq (4P{^\sharp}). 
\end{equation*}

\item
\label{IC-(8.8.2)}
$H^0(\OOO_F(-\nu C)) \twoheadrightarrow H^0(\OOO_C
\totimes\OOO_F(-\nu C))$ for all $\nu \geq 0$.

\item
\label{IC-(8.8.3)}
There are sections
$s_1$, $s_2 \in H^0(I_C)$ such that 
\begin{equation*}
\begin{array}{rll} 
s_1&\equiv (\unit)\cdot (y_1+\xi_1 y_2^{m-1})^2(y^{m-2}_1-y^2_2) 
&\mod y_4\quad \text{\rm near $P$}, 
\\[5pt]
s_2&\equiv (\unit)\cdot (y_2+\xi_2y_1^{m-1})(y^{m-2}_1-y^2_2)^{(m-1)/2} 
&\mod y_4\quad \text{\rm near $P$}, 
\end{array} 
\end{equation*} 
where $\xi_1$, $\xi_2\in \OOO_{X^\sharp}$ are invariants. 

\item
\label{IC-(8.8.4)}
$H^0(I_C) \twoheadrightarrow H^0(\gr^0_C J)=H^0(I_C/\F^3(\OOO,J)) \simeq \CC$.
\end{enumerate}
\end{thm}

\begin{proof} \xref{IC-(8.8.1)} follows from the construction of $F$. Hence, 
$H^1(\OOO_C\totimes\OOO_F(-\nu C))=0$ for all $\nu 
\geq 0$, and $H^1(\OOO_F(-\nu C))=0$ since $C$ 
is a fiber of proper $f$. Thus we have \xref{IC-(8.8.2)}.

To prove \xref{IC-(8.8.3)} 
consider the Stein factorization 
\eqref{IC=equation-Stein-factorization} and as in the proof of Lemma \ref{lemma-IC-H}
we take an embedding  $(F_Z,o_Z)\subset\CC^3_{x,y,z}$ so that $(F_Z,o_Z)$ is given 
by the equation $z^2+xy^2+x^{m-1}$ and
the map $f_2:(F_Z,o_Z)\to (Z,o)$ is just the projection to the $(x,y)$-plane.
Take $s_1=f^*x$ and $s_2=f^*y$. The weighted blowup of $(F_Z,o_Z)$ with weights 
$(2,m-2,m-1)$ extracts the central vertex of the \type{D_m}-diagram \eqref{IC-eq-Dn-diagram}.
The multiplicity of the corresponding exceptional curve in $f_2^*x$ and $f_2^*y$ is equal to
$2$ and $m-2$, respectively. Using this one can easily show that 
multiplicities of all exceptional curves in $f_2^*x$ and $f_2^*y$,
respectively, are given by the following 
diagrams 
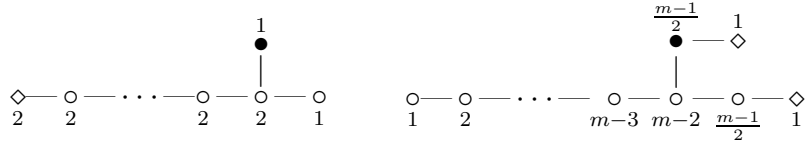
\[
\begin{array}{c@\,c@\,c@\,c@\,c@\,c@\,c}
&&&&\overset1\bullet
\\[-4pt]
&&&&|
\\[-4pt]
\underset2\diamond\lin&\underset2\circ&\lin\cdots\lin
\underset2\circ&\lin&\underset2\circ&\lin&\underset1\circ
\end{array} 
\hspace{19pt}
\begin{array}{c@\,c@\,c@{\hspace{-4pt}}c@{\hspace{-3pt}}c@{\hspace{-3pt}}c@{\hspace{-4pt}}c@{\hspace{-4pt}}c@\,c}
&&&&\overset{\frac{m-1}{2}}\bullet&\lin&\overset{1}\diamond
\\[-2pt]
&&&&|
\\[-4pt]
\underset1\circ\lin&\underset2\circ&\lin\cdots\lin
\underset{m-3}\circ&\lin&\underset{m-2}\circ&\lin&\underset{\frac{m-1}{2}}\circ&\lin& \underset{1}\diamond
\end{array} 
\]
where the vertex $\bullet$, as usual, corresponds to $C$ and vertices $\diamond$ correspond
to components of the proper transforms of $\{f_2^*x=0\}$ and $\{f_2^*y=0\}$.
The multiplicity of $C$ is exactly the exponent of $y^{m-2}_1-y^2_2$ in $s_i$ $\mod y_4$. 
Therefore, 
\[
s_1\equiv \gamma_1(y^{m-2}_1-y^2_2) \qquad 
s_2\equiv \gamma_2(y^{m-2}_1-y^2_2)^{(m-1)/2} \mod y_4, 
\]
where
$\gamma_i\in \OOO_{X^\sharp}$ are semi-invariants.
Using the above diagrams, we see
$(\{\gamma_1=0\} \cdot C)_F=-4/m$ and $(\{\gamma_2=0\} \cdot C)_F=(m-2)/m$
because $(C^2)_F=4/m$ by (i).
Since $y_1y_2$ is of weight $0$, we have 
\[
\gamma_1 = (\unit) \cdot (y_1+y_2^{m-1}\xi_1)^2 \mod y_4
\]
for some $\xi_1 \in \OOO_X$.
Indeed, since $\gamma_1=0$ defines a double curve on $F$, 
one has $\gamma_1=(\unit)\cdot \delta^2 \mod y_4$
for some $\delta \in \OOO_{X^{\sharp}}$ with weight $\equiv 2$ such that
$\delta|_C=y_1|_C$.

Similarly, we have
$\gamma_2|_C= y_2|_C$.
Hence,
\[
\gamma_2 =(\unit) \cdot (y_2+y_1^{m-1}\xi_2) \mod y_4.
\]

Finally, \xref{IC-(8.8.4)} follows from \xref{IC-(8.8.3)} because $H^0(\gr^0_C J) \simeq \CC$. 
\end{proof}

\begin{de}
\label{IC-(8.9)}
By Proposition \xref{IC-Proposition-(8.7)} there are four cases to treat. 
\begin{sde}
\label{IC-(8.9.1)}
{\bf Case} $m \geq 7$, $\alpha (P) \neq 0$. 
\end{sde}

\begin{sde}
\label{IC-(8.9.2)}
{\bf Case} $m=5$, $\lambda_1(P) \neq 0$. 
\end{sde}

\begin{sde}
\label{IC-(8.9.3)}
{\bf Case} $m=5$, $\lambda_1(P)=0$, $\alpha (P) \neq 0$.
 \end{sde}

\begin{sde}
\label{IC-(8.9.4)}
{\bf Case} $m=5$, $\lambda_1(P)=0$, $\alpha (P)=0$. 
\end{sde}

We shall show that cases \ref{IC-(8.9.1)}, \ref{IC-(8.9.2)}, \ref{IC-(8.9.3)}
do not occur and \ref{IC-(8.9.4)} implies \ref{IC-(8.3.2)}.
\end{de}

\begin{de}{\bf Proof of \xref{IC-main}; cases \xref{IC-(8.9.1)} and \xref{IC-(8.9.3)}.}
\label{IC-(8.10)}
By \eqref{IC-eq-(8.5.2)} 
and Proposition \xref{IC-(8.8)}, a general section $s \in H^0(I_C)$ satisfies 
\begin{equation*} 
s \equiv (\unit)\cdot \bigl(y^2_1u +\alpha y_2y^2_4\bigr) \mod \F^3(\OOO,J)\quad \text{\rm at $P$}, 
\end{equation*} 
where $\alpha (P) \neq 0$ by assumption. Let us take $s_2$ given 
in \xref{IC-(8.8.3)} of Proposition \ref{IC-(8.8)}. 
We claim that $s_2$ belongs to 
$H^0(\F^3(\OOO,J))$. Indeed, it is obvious that 
$s \notin \CC\cdot s_2 + \F^3(\OOO,J)$ near $P$. Hence by 
$H^0(I_C/\F^3(\OOO,J))=\CC\cdot s$, we have 
$s_2 \in H^0(\F^3(\OOO,J))$ as claimed. By Lemma \xref{IC-(8.5.3)}, we see that the coefficient 
of $y_2y^2_4$ (resp. $y^m_2$) in the Taylor expansion of $s_2$ 
at $P{^\sharp}$ is $0$ (resp. 
non-zero) because $m \geq 7$ or $\lambda_1(P)=0$. We now 
analyze the set $H=\{s=0\}$. 
By Bertini's theorem, $H$ is smooth outside of $C$. Since $\OOO\cdot s$ 
is the unique $\OOO$ in $\gr^1_C\OOO \simeq \OOO\oplus \OOO(-1)$, $H$ is smooth on $C\setminus\{P\}$. To 
study $(H,P)$, we can apply \cite[10.7]{Kollar-Mori-1992}. 
Indeed, if $\lambda_1(P)=0$, 
then $\mu_1(P) \neq 0$ by the construction \xref{IC-(8.2)}. Thus \cite[10.7.1]{Kollar-Mori-1992} holds by 
Lemma 
\xref{IC-(8.5.3)}. Replacing $s$ with a general linear combination of 
$s$ and $s_2$ we see that 
\cite[10.7.2]{Kollar-Mori-1992} is satisfied. Since $m \geq 7$ or 
$\lambda_1(P)=0$, we can now apply \cite[10.7]{Kollar-Mori-1992}.
One can see that the contraction $f_H:H\to T$ must be birational in this case,
which is a contradiction. 
\end{de}

\begin{de}{\bf Proof of \xref{IC-main}; case \xref{IC-(8.9.2)}.}
\label{IC-(8.11)}
The argument is the same as \xref{IC-(8.10)} except that we need to check 
the conditions of \cite[10.7]{Kollar-Mori-1992}. 
Note that \eqref{IC-IC-notation-equation-u} has the form  
$u=\lambda_1y_4+\mu_1(y_1^3-y_2^2)$.
Since $\lambda_1(P) \neq 0$, by a coordinate change we can make 
$\mu_1(P) \neq 0$. 
Let $D:=\{y_1=0\}/\muu_m \in |-2K_X|$ and let 
\begin{equation*} 
\phi_D:=\frac{u-\lambda_1(P)y_4}{\dd y_1\wedge \dd y_2\wedge \dd y_4} =
 \frac{(\lambda_1-\lambda_1(P))y_4 + \mu_1(y^3_1-y^2_2)}{\dd y_1\wedge \dd y_2\wedge \dd y_4}
\in \OOO_D(-K_X). 
\end{equation*}
Arguments in \cite[3.1]{Mori-Prokhorov-2008III} show that there exists 
a section $\phi \in H^0(\OOO(-K_X))$ sent to $\phi_D$ modulo $\omega_Z$.
Thus the image of $\phi$ under the homomorphism 
\begin{equation*} 
I_C\totimes\OOO_X(-K_X) \twoheadrightarrow 
\gr^1_C\OOO_X(-K_X)=(1) \toplus (0) 
\twoheadrightarrow (0) 
\end{equation*} 
is non-zero because $\lambda_1(P) \neq 0$. 
Hence $F'=\{\phi=0\} \in |-K_X|$ is smooth 
outside of $P$ and we may choose $\phi$ so that $F'$ is 
furthermore normal by Bertini's theorem. We have an $\ell$-splitting 
\begin{equation*} 
\gr^1_C\OOO=(4P{^\sharp}) \toplus \OOO_C(-F'). 
\end{equation*} 
By the construction of $F'$, we see that 
$(F',P)=\{v=0\}/\muu_m$, where 
$v=y^3_1-y^2_2+\lambda_1' y_4$ for some 
$\lambda_1' \in \OOO_{C,P}$ such that $\lambda_1'(P)=0$. As in Proposition 
\xref{IC-(8.8)}, we see that $\OOO_{F'}(-C)$ is an 
$\ell$-invertible $\OOO_{F'}$-module with an $\ell$-free 
$\ell$-basis $u$ at $P$ and there exists an $\ell$-isomorphism 
\begin{equation*} 
\OOO_C\totimes\OOO_{F'}(-C) \simeq (4P{^\sharp}). 
\end{equation*} 
We similarly see 
\begin{equation*} 
H^0(\OOO_{F'}(-\nu C)) 
\twoheadrightarrow 
H^0(\OOO_C\totimes\OOO_{F'}(-\nu C))\quad 
\text{for all $\nu \geq 0$}. 
\end{equation*} 
We note that $y^2_1u$ and $y_2u^2$ are bases of 
$\OOO_C\totimes\OOO_{F'}(-\nu C)$ at $P$ for $\nu=1$ and $2$, respectively. Thus, for 
arbitrary $a_1,\, a_2 \in \CC$, there exist 
a section $s_0'\in H^0(\OOO_{F'}(-C))$ such that
\begin{equation*} 
s_0' \equiv a_1y^2_1u + a_2y_2u^2 \mod (v,u^3). 
\end{equation*} 
Recall that the map $H^0(\OOO_X)\to H^0(\OOO_{F'})$ is surjective 
modulo $f^*\omega_Z$ \cite[Proposition 2.1]{Mori-Prokhorov-2008III}.
In our situation, sections of $f^*\omega_Z$ lifted to $\CC^3_{y_1,y_2,y_4}$
are contained into $\wedge^2\Omega^1_X$.
We claim 
\begin{equation}
\label{IC-inclusion-omega}
\bigwedge\limits^2\Omega^1_X
 \subset (y_1, y_2, y_4)^3 \cdot \Omega_{X^\sharp}^2
\subset (y_1, y_2, y_4)^4\cdot \omega_{F'^\sharp}.
\end{equation}
on the index-one cover $F'^\sharp \subset X^\sharp$ of $F' \subset X$.

Note first that the local coordinates of $X$ at $P$ are 
\[
y_1y_2,\quad y_1^5,\quad y_2^5,\quad y_1^2y_4,\quad y_2^3y_4,\quad y_2y_4^2.
\]
Since $y_1y_2$ is the only term of degree 2, and the
rest are of degree $\ge 3$,
we see that $\wedge^2 \Omega_X^1 \subset (y_1,y_2,y_4)^3\cdot\Omega_{X^\sharp}^2$,
the first inclusion.

Since $\phi=\beta_1(y_1^3-y_2^2)+ \beta_2 y_4$ with $\beta_1,\, \beta_2 \in \OOO_X$ such that
$\beta_2(P)=0$,
we have $\Omega_{X^\sharp}^2 |_{F'^\sharp}\subset (y_1, y_2, y_4)\cdot \omega_{F'^\sharp}$
because
\[
\Omega:=
\left.\frac{\dd y_2\wedge \dd y_4}{\partial \phi/\partial y_1}\right|_{F'^\sharp}=
\left.\pm \frac{\dd y_1\wedge \dd y_4}{\partial \phi/\partial y_2}\right|_{F'^\sharp}=
\left.\pm \frac{\dd y_1\wedge \dd y_2}{\partial \phi/\partial y_4}\right|_{F'^\sharp}
\in \omega_{F'^\sharp},
\]
which settles the second inclusion.

From \eqref{IC-inclusion-omega} and $(v,u^3) \subset (y_1^3,y_2^2,y_4^3)$
we see that 
there exists $s' \in H^0(I_C)$ such that
\[
s' \equiv a_1y_2y_4+a_2y_2y_4^2 \mod (y_1,y_2,y_4)^4+(y_1^3,y_2^2,y_4^3).
\]
By this, we obtain non-vanishing of the coefficient of $x_2x_3^2$
in \cite[10.7]{Kollar-Mori-1992}.
Note that \cite[10.7.1]{Kollar-Mori-1992}
is satisfied because $\lambda_1(P)\neq 0$ and
\cite[10.7.3]{Kollar-Mori-1992} is satisfied because 
the term $y_2^5$ appears and $y_1^2y_2^2$ does not appear in $s_2$.
The rest is the same as \xref{IC-(8.10)}.

\begin{sde}
{\bf Remark.}
In \cite{Kollar-Mori-1992}, the explanation at the beginning of \cite[8.11]{Kollar-Mori-1992} 
was not appropriate; the non-vanishing of the coefficient of 
$x_2x_3^2$ of \cite[10.7]{Kollar-Mori-1992} as well as \cite[10.7.3]{Kollar-Mori-1992} should
have been verified.
The last three lines of our \ref{IC-(8.11)} supplements the
insufficient treatment in \cite[8.11]{Kollar-Mori-1992}.
\end{sde}
\end{de}

\begin{de} {\bf Case \xref{IC-(8.9.4)}. } 
Then $m=5$ and $\lambda_1(P)=\alpha (P)=0$.
Since $\lambda_1(P)=0$, 
we have $\mu_1(P) \neq 0$ because $u$ is an $\ell$-basis 
(see \eqref{IC-IC-notation-equation-u}). Since $\alpha (P)=0$, we have 
$\alpha y_2=\lambda_2y^4_1$ 
for some $\lambda_2 \in \OOO_{C,P}$ as in Lemma \xref{IC-lemma-l-split-(8.6)}. Thus a general 
section $s \in H^0(I_C)$ 
satisfies the following relation near $P$: 
\begin{equation}
\label{IC-eq-(8.12.1)} 
s \equiv (\unit)\cdot y^2_1(u+\lambda_2y^2_1y^2_4) \mod \F^3(\OOO,J).
\end{equation} 
Hence $s$ does not contain any of the terms $y_1y_2$, $y_1^2y_4$, $y_2y_4^2$
and contains terms $y_1^5$, $y_1^2y_2^2$.
By the lemma below $s$ contains also $y_2^3y_4$.

\begin{sthm}{\bf Lemma.}\label{IC-last-lemma-1}
Let $\tau$ be the weight $\tau=\frac15(4,1,2)$ 
and let 
$(H,P)\subset \CC^3/\muu_5(2,3,1)$ be a normal surface singularity given by 
$\phi(x_1,x_2,x_3)=0$, where $\phi$ is a
$\muu_5$-invariant that does not contain any terms of $\tau$-weight
$<2$.
Then $(H,P)$ is not a rational singularity. 
\end{sthm}

\begin{proof}
According to \cite{Elkik-1978} we may assume that the coefficients of
$\phi$ are general under the assumption $\phi_{\tau=1}=0$.
Consider the weighted blowup with weight $\tau$.
The exceptional divisor $\Upsilon$ is given in $\PP(4,1,2)$ by the equation
$\phi_{\tau=2}(x_1,x_2,x_3)=0$
or, equivalently, in $\PP(2,1,1)$ by $\phi_{\tau=2}(x_1,x_2^{1/2},x_3)=0$.
Thus, $\Upsilon\in |\OOO_{\PP(2,1,1)}(5)|$ is a general member.
By Bertini's theorem $\Upsilon$ is smooth
and the pair $(\PP(2,1,1),\Upsilon)$ is PLT.
By the subadjunction formula
\[
2p_a(\Upsilon)-2=(K_{\PP(2,1,1)}+\Upsilon) \cdot \Upsilon-\frac12=2.
\]
Hence, $\Upsilon$ is not rational.
\end{proof}

\begin{sthm}{\bf Lemma.}\label{IC-last-lemma-2}
The equation $s$ contains the term $y_1y_4^3$. 
\end{sthm}
\begin{proof}
Since $\alpha(P)=0$, we can write $\alpha=y_1y_2\beta$ for some $\beta\in \OOO_{C,P}$.
The unique $\OOO\subset \gr_C^0J$ is generated near $P$ by 
\[
y_1^2u+(y_1y_2\beta)y_2y_4^2=
y_1^2u+y_1^4\beta y_4^2=
y_1^2(u+y_1\beta y_4^2) \in \F^3(\OOO,J). 
\]
By Lemma \ref{IC-lemma-l-split-(8.6)} the sequence \eqref{IC-eq-(8.5.1)} splits and we have 
\[
\begin{array}{c@{\hspace{1pt}}c@{\hspace{-7pt}}c}
\gr^0_C J \simeq (4P{^\sharp}) &\toplus & \OOO_C(-2F)
\\[-2.5pt]
&&\| 
\\[-2.5pt]
&&(-1+(3P^\sharp)).
\end{array}
\]
Let $\KKK$ be the $C$-laminal 
ideal such that $J \supset \KKK \supset \F^3(\OOO_C,J)$ and $\KKK/\F^3(\OOO,J)=(4P{^\sharp})$. 
Then $\KKK$ is locally a nested c.i. on $C\setminus\{P\}$ and 
$(y_4,u)$ is a $(1,3)$-monomializable $\ell$-basis of 
$I_C\supset \KKK$ at $P$ (where $u$ is given by \eqref{IC-IC-notation-equation-u}).
We have 
\[
\begin{array}{ccl}
0\to &(-1+2P^\sharp) & \longrightarrow \gr_C^0\KKK 
\longrightarrow (4P^\sharp) \to 0
\\[-2.5pt]
&\|
\\[-2.5pt]
&\OOO_C(-3F)
\end{array}
\]
Since $H^1(\OOO_C(-3F)\totimes \omega)\neq 0$, as in the proof of Proposition 
\ref{IC-Proposition-(8.7)} the sequence does not split.
So, locally near $P$, the sheaf $\gr_C^0\KKK$ has a section 
$y_1^2u+\gamma y_1y_4^3$ with $\gamma(P)\neq 0$.
\end{proof}

Thus, by the two lemmas
\ref{IC-last-lemma-1} and \ref{IC-last-lemma-2} above, $s$ does not contain any of the terms 
$y_1y_2$, $y_1^2y_4$, $y_2y_4^2$
and contains terms $y_1^5$, $y_1^2y_2^2$, $y_2^3y_4$, $y_1y_4^3$.
Therefore, \cite[10.8]{Kollar-Mori-1992} can be applied to $(H,P)$.
It is easy to see that the whole configuration contracts to a curve. 
We get \ref{IC-(8.3.2)}. This completes the proof of Theorem \xref{IC-main}.
\end{de}

\section{Case \type{(IIB)}}

\begin{de}{\bf Setup.}
\label{IIB-local-description}
Let $(X,P)$ be the germ of a three-dimensional terminal singularity and let $C\subset (X,C)$ be 
a smooth curve. Recall that the triple $(X,C,P)$ is said to be of type \type{(IIB)} if 
$(X,P)$ is a terminal singularity of type \type{cAx/4} and
there are analytic isomorphisms
\[
(X,P) \simeq \{y_1^2-y_2^3+\alpha =0\}/\muu_4(3,2,1,1)\subset \CC^4_{y_1,\dots,y_4} /\muu_4(3,2,1,1),
\]
\[
C\simeq \{y_1^2-y_2^3=y_3=y_4=0\}/\muu_4 (3,2,1,1),
\]
where $\alpha=\alpha(y_1,\dots,y_4)\in (y_3,\, y_4)$ 
is a semi-invariant with $\wt \alpha\equiv 2\mod 4$ and 
$\alpha_2(0,0,y_3,y_4)\neq 0$ (see \cite[A.3]{Mori-1988}).
\begin{sde}{\bf Definition.}\label{IIB-local-definition}
We say that $(X,P)$ is a \textit{simple} (resp. \textit{double}) \type{cAx/4}-point if $\rk \alpha_2(0,0,y_3,y_4)=2$
(resp. $\rk \alpha_2(0,0,y_3,y_4)=1$). 
\end{sde}

\begin{sde} 
\label{IIB-notation-1}
Let $(X,C)$ be an extremal curve germ and let $f\colon (X, C) \rightarrow (Z, o)$ be 
the corresponding contraction. 
In this section we assume that $C$ is irreducible and has a point 
$P$ of type \type{(IIB)}.
According to \cite[Theorem 4.5]{Kollar-Mori-1992} the germ $(X,C)$ is not flipping.
Recall that $(X,C)$ is locally primitive at $P$ \cite[4.2]{Mori-1988}.
Moreover, $P$ is the only singular point on \cite[Theorem 6.7]{Mori-1988}, 
\cite[Theorem 8.6, Lemma 7.1.2]{Mori-Prokhorov-2008}.
Thus the group $\Cl (Z,o)$ has no torsion.
Therefore, $f$ is either a divisorial contraction to a cDV point
or a conic bundle over a smooth base. 
\end{sde}
\end{de}

\begin{de}\label{IIB-ge}
According to \cite[Theorem 2.2]{Kollar-Mori-1992} and 
\cite{Mori-Prokhorov-2008III} a general member $F\in |-K_X|$ contains $C$, has only Du Val singularities,
and the graph $\Delta(F,C)$ has the form
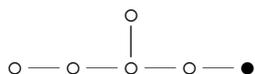
\[
\begin{array}{c@{\,}c@{\,}c@{\,}c@{\,}c@{\,}c@{\,}c@{\,}c@{\,}c}
&&&&\circ
\\[-4pt]
&&&&|
\\[-4pt]
\circ&\lin&\circ&\lin&\circ&\lin&\circ&\lin&\bullet
\end{array}
\]
where all the vertices correspond to $(-2)$-curves and $\bullet$ corresponds to $C$.
Under the identifications
of \ref{IIB-local-description}, 
a general member $F\in |-K_X|$ near $P$ is given by $\lambda y_3+\mu y_4=0$
for some $\lambda,\, \mu \in \OOO_X$
such that $\lambda(0)$, $\mu(0)$ are general in $\CC^*$
\cite[2.11]{Kollar-Mori-1992}, \cite[\S 4]{Mori-Prokhorov-2008III}.
\end{de}

\begin{de}
Let $H$ be a general member of $|\OOO_{X}|_C$, 
let $T:=f(H)$, and let $\Gamma:=H\cap F$. 

\begin{sde}
If $f$ is divisorial, we put $F_Z:=f(F)$ and $\Gamma_Z:=f(\Gamma)$. 
Then $F_Z\in |-K_Z|$, $T$ is a general hyperplane section of 
$(Z,o)$ and $\Gamma_Z$ is a general hyperplane section of $F_Z$. 
\end{sde}
\begin{sde}\label{IIB-Stein-factorization}
If $f$ is a $\QQ$-conic bundle, we consider 
the Stein factorization 
\[
f_F: (F,C)\stackrel{f_1}\longrightarrow (F_Z, o_Z) \stackrel{f_2}\longrightarrow (Z,o).
\]
Here we put $\Gamma_Z:=f_1(\Gamma)$.
\end{sde}

In both cases $F_Z$ is a Du Val singularity of type \type{E_6} by \ref{IIB-ge}.
\end{de}

\begin{thm} {\bf Lemma.}
\label{lemma-IIB-H}
\begin{enumerate}
 \item 
$H$ is normal, has only rational singularities, and smooth outside of $P$\textup;

 \item 
$\Gamma=C+\Gamma_1$ \textup(as a scheme\textup), where $\Gamma_1$ is a reduced irreducible curve\textup;
 \item 
if $f$ is birational, then $T=f(H)$ is Du Val singularity of type 
\type{E_6}, \type{D_5}, \type{D_4}, \type{A_4},\dots,\type{A_1}
\textup(or smooth\textup).
\end{enumerate}
\end{thm}
\begin{proof}
Consider two cases:

\begin{sde}\label{case-IIB-H-divisorial}
{\bf Case: $f$ is divisorial.}
Since the point $(Z,o)$ is terminal of index $1$, the germ $(T,o)$ is a Du Val singularity.
Since $\Gamma_Z$ is a general hyperplane section of $F_Z$,
we that the graph $\Delta(F,\Gamma)$ has the form
 \begin{equation}
\label{equation-IIB-graph-E6}
\begin{array}{c@{\,}c@{\,}c@{\,}c@{\,}c@{\,}c@{\,}c@{\,}c@{\,}c}
&&&&\diamond
\\[-5pt]
&&&&|
\\[0pt]
&&&&\overset{2}\circ
\\[-5pt]
&&&&|
\\[-4pt]
\underset{1}\circ&\lin&\underset{2}\circ&\lin&\underset{3}\circ&\lin&\underset{2}\circ&\lin&\underset{1}{\bullet}
\end{array}
 \end{equation}
where, as usual, $\diamond$ corresponds to the proper transform of $\Gamma_Z$
and numbers attached to vertices are coefficients of corresponding exceptional curves in the pull-back of $\Gamma_Z$.
By Bertini's theorem $H$ is smooth outside of $C$.
Since the coefficient of $C$ equals to $1$, $F\cap H=C+\Gamma$ (as a scheme), so $H$ is smooth 
outside of $P$. In particular, $H$ is normal.
Since $f_H: H\to T$ is a birational contraction and $(T,o)$
is a Du Val singularity, the singularities of $H$ are rational.

\begin{sde}\label{case-IIB-H-conic-bundle}
{\bf Case: $f$ is a $\QQ$-conic bundle.}
We may assume that, in some coordinate system, 
the germ $(F_Z, o_Z)$ is given by
$x^2+y^3+z^4=0$. Then by \cite{Catanese-1987}
up to coordinate change the double cover $(F_Z, o_Z) \longrightarrow (Z,o)$
is just the projection to the $(y,z)$-plane.
Hence we may assume that $\Gamma_Z$ is given by $z=0$.
As in the case \ref{case-IIB-H-divisorial} we see that the graph $\Delta(F,\Gamma)$ has the form
\eqref{equation-IIB-graph-E6}. Therefore, $H$ is smooth outside of $P$.
The restriction $f_H: H\to T$ is a rational curve fibration.
Hence $H$ has only rational singularities.
\end{sde}
(iii) follows by the fact that there is a hyperplane section $F_Z$ of $(Z,o)$
which is 
Du Val of type \type{E_6} (see e.g. \cite{Arnold1972}).
\end{sde}
\end{proof}

We need a more detailed description of $(H,C)$ near $P$.
\begin{sthm} {\bf Lemma.}\label{II-B-lemma-equation-HH}
In the notation of \xref {IIB-local-description} the surface 
$H\subset X$ is locally near $P$ given by 
the equation $y_3 v_3+ y_4 v_4=0$, where $v_3,\, v_4\in \OOO_{P^\sharp, X^\sharp}$
are semi-invariants with $\wt v_i\equiv 3$
and at least one of $v_3$ or $v_4$ contains a linear term in $y_1$.
\end{sthm}

\begin{proof}
Since $H$ is normal and $\gr_C^1\OOO\simeq \OOO_{\PP^1}\oplus \OOO_{\PP^1}(-1)$,
we have $\OOO_C(-H)=\OOO\subset \gr_C^1\OOO$, i.e. the local equation of $H$
must be a generator of $\OOO\subset \gr_C^1\OOO$.
\end{proof}

\begin{de}
Let $\sigma$ be the weight $\frac14(3,2,1,1)$. By Lemma \ref{II-B-lemma-equation-HH}
the surface germ $(H,P)$ can be given in $\CC^4/\muu_4(3,2,1,1)$ by 
two equations:
\begin{equation}
\label{equation-IIB-H}
\begin{cases}
 y_1^2-y_2^3+\eta(y_3,y_4)+ \phi(y_1,y_2,y_3,y_4)=0,
\\[4pt]
y_1l(y_3,y_4)+y_2q(y_3,y_4)+\xi(y_3,y_4)+\psi(y_1,y_2,y_3,y_4) =0,
\end{cases}
\end{equation}
where $\eta$, $l$, $q$ and $\xi$ are homogeneous polynomials 
of degree $2$, $1$, $2$ and $4$, respectively, $\eta\neq 0$, $l\neq 0$,
$\phi, \, \psi \in (y_3,\, y_4)$, $\sigma\mbox{-}\ord \phi\ge 3/2$, $\sigma\mbox{-}\ord \psi\ge 2$.
Moreover, $\rk \eta=2$
(resp. $\rk \eta=1$) if $(X,P)$ is a simple (resp. double) \type{cAx/4}-point.

\begin{sde}
Consider the weighted blowup 
\[
g: (W\supset \tilde X\supset \tilde H) \longrightarrow (\CC^4/\muu_4(3,2,1,1)\supset X\supset H)
\]
with weight $\sigma$.
Let $E$ be the $g$-exceptional divisor, let $\Xi:=E\cap \tilde H$
be the exceptional divisor of $g_H:=g|_{\tilde H}$,
and let $\tilde C$ be the proper transform of $C$.
Denote
\[
\Xi_0:=\{y_3=y_4=0\}\subset E. 
\]
If $\tilde H$ is normal, let $g_1:\hat H\to \tilde H$ be the minimal resolution. 
Thus, in this case, we have the following morphisms
\[
h: \hat H \overset{g_1}\longrightarrow \tilde H \overset{g_H}\longrightarrow H 
\overset{f_H}\longrightarrow T.
\]

\begin{sthm}{\bf Lemma.}\label{IIB-lemma-computations-self-intersections-1}
\begin{enumerate}
 \item 
$E\simeq \PP(3,2,1,1)$ and $\Xi$ is given in this $\PP(3,2,1,1)$ by
\[
\eta(y_3,y_4)=y_1l(y_3,y_4)+y_2q(y_3,y_4)+\xi(y_3,y_4)=0;
\]
 \item 
$\tilde C$ of $C$ meets $E$ at $Q:=(1:1:0:0)\in \Xi_0$;
 \item \label{IIB-computations-self-intersections-1}
$\Xi_0$ is a component of $\Xi$ and 
$(\Xi_0\cdot \Xi)_{\tilde H}=-2/3$;
 \item \label{equation-IIB-discrepansies}
If $\tilde H$ is normal, then 
$K_{\tilde H}=g^*K_H-\frac{3}{4}\Xi$.
\end{enumerate}
\end{sthm}
\begin{proof}
Statements (i) and (ii) are obvious, (iii) follows from 
\begin{equation*}
\qquad (\Xi_0\cdot \Xi)_{\tilde H}=
(\Xi_0\cdot E)_{W}
=(\Xi_0\cdot \OOO_E(E))_E = 
(\Xi_0\cdot \OOO_E(-4))_E=
-\frac23,
\end{equation*}
and (iv) follows from 
$K_W=g^*K_{\CC^4/\muu_4}+ \frac 34E$. 
\end{proof}
\end{sde}
\end{de}

\begin{de}{\bf Case of simple \type{cAx/4}-point.}
After a coordinate change we may assume that $\eta=y_3y_4$. 
We also may assume that 
the term $y_3$ appears in $l(y_3,y_4)$ with coefficient $1$, that is,
$l(y_3,y_4)=y_3+cy_4$, $c\in \CC$.
Thus the equations \eqref{equation-IIB-H} for $(H,P)$ have the form:
\begin{equation}
\label{equation-IIB-rank=2}
\begin{cases}
 y_1^2-y_2^3+ y_3y_4+\phi=0,
\\[4pt]
y_1(y_3+cy_4)+y_2q(y_3,y_4)+\xi(y_3,y_4)+\psi =0.
\end{cases}
\end{equation}
It is easy to see that in this case $\tilde X$ has only isolated (terminal) singularities.
Indeed, $\tilde X\cap E$ is given by $y_3y_4=0$ in $E\simeq\PP(3,2,1,1)$.
Hence, $\Sing (\tilde X)\subset \Xi_0\cup \Sing(E)$.
There are the following subcases.

\begin{sde}
{\bf Subcase: $(X,P)$ is simple \type{cAx/4}-point and $c\neq 0$.}\label{II-B-case-most-general}
We shall show that only the case \ref{IIB-theorem-A2case-simple} occurs.
We may assume that in \eqref{equation-IIB-rank=2} $l(y_3,y_4)=y_3+y_4$.
In this case, $\Xi=2\Xi_0+\Xi'+\Xi''$, where $\Xi'$ and $\Xi''$ 
are given in $E\simeq \PP(3,2,1,1)$ as follows:
\[
\begin{array}{ll}
\Xi':=&\{y_3=y_1+y_2q(0,y_4)/y_4+\xi(0,y_4)/y_4=0\},
\\[4pt]
\Xi'':=&\{y_4=y_1+y_2q(y_3,0)/y_3+\xi(y_3,0)/y_3=0\}.
\end{array}
\] 
All the components of $\Xi$ pass through $(0:1:0:0)$ and do not meet each other 
elsewhere.

\begin{ssthm}{\bf Claim.}\label{IIB-claim-1}
The surface $\tilde H$ is normal and has the following singularities
\textup(in natural weighted coordinates on $E\simeq \PP(3,2,1,1)$\textup)\textup:
\begin{itemize}
\item 
$O_1:=(1:0:0:0)$ which is of type \type{A_2},
\item 
$Q:=\Xi_0\cap \tilde C=(1:1:0:0)$ which is of type \type{A_1},
\item 
$O_2:=\Xi_0\cap \Xi'\cap \Xi''=(0:1:0:0)$ which is a log terminal point of index $2$
\textup(a cyclic quotient singularity of type $\frac{1}{4k}(1,2k-1)$\textup).
\end{itemize}
Pairs $(\tilde H,\Xi_0+\Xi'+\tilde C)$ and $(\tilde H,\Xi_0+\Xi''+\tilde C)$ are 
log canonical \textup(LC\textup).
Moreover, they are purely log terminal \textup(PLT\textup) at all points of 
$\Xi_0\setminus \{O_2,\, Q\}$. 
Thus the surface $\tilde H$ looks as follows\textup: 
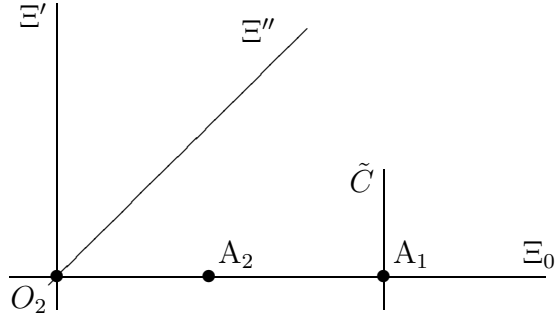
\[
\xy
(48,0)="A" *{}*+!DR{\tilde C},
(48,-10)="B" *{}*+!DR{},
(5,-10)="C" *{}*+!DR{},
(5,22)="D" *{}*+!DR{\Xi'},
(3,-10)="E" *{}*+!DL{},
(65,-10)="F" *{}*+!DL{\Xi_0},
(25,-10)="EE" *{\bullet}*+!DL{\mathrm A_2},
(7,-8)="X" *{}*+!DL{},
(35,20)="Y" *{}*+!DR{\Xi''},
(25,15)="TT" *{}*+!DL{},
(48,-10)="FF" *{\bullet}*+!DL{\mathrm A_1},
{"E";"F":"C";"D",x} ="I" *{\bullet}*+!UR{O_2},
"B";"A"**{} +/1pc/;-/1pc/ **@{-},
"C";"D"**{} +/1pc/;-/1pc/ **@{-},
"E";"F"**{} +/1pc/;-/1pc/ **@{-},
"X";"Y"**{} +/1pc/;-/1pc/ **@{-}
\endxy
\]
\end{ssthm}

\begin{proof}
Since $\Xi=\tilde H\cap E$ is reduced along $\Xi'$ and $\Xi''$,
the singular locus of $\tilde H$ is contained in $\Xi_0=\{y_3=y_4=0\}$.

Consider the chart $U_1=\{y_1\neq 0\}\subset W$,\ $U_1\simeq \CC^4/\muu_3(2,2,1,1)$. 
The equations of $\tilde H$ have the form 
\[
\begin{cases}
 y_1-y_1y_2^3+ y_3y_4+y_1\phi_{3/2}(1,y_2,y_3,y_4)+y_1^2(\cdots)=0,
\\[4pt]
y_3+y_4+y_2q(y_3,y_4)+\xi(y_3,y_4)+y_1\psi_{2}(1,y_2,y_3,y_4)+y_1^2(\cdots)=0.
\end{cases}
\]
and $\tilde C$ is cut out on $\tilde H$ by $y_3=y_4=0$.
Using the condition $y_1=y_3=y_4=0$ one can obtain that 
the surface $\tilde H\cap U_1$ 
has on the exceptional divisor $\{y_1=0\}$ two singular points: 
$Q=\{y_1=y_3=y_4=1-y_2^3=0\}$ and the origin $O_1$.
It is easy to see that $(\tilde H,\,Q)$ is a Du Val singularity of type \type{A_1}
and $(\tilde H,\,O_1)$ is a Du Val singularity of type \type{A_2}.
Since $\Xi_0$ and $\tilde C$ are 
smooth curves meeting each other transversely, the pair $K_{\tilde H}+\Xi_0+\tilde C$
is LC at $Q$.

Consider the chart $U_{2}=\{y_2\neq 0\}\subset W$, $U_2\simeq \CC^4/\muu_2(1,0,1,1)$. 
The equations of $\tilde H$ have the form 
\[
\begin{cases}
 y_1^2y_2-y_2+ y_3y_4+y_2\phi_{3/2}(y_1,1,y_3,y_4)+y_2^2(\cdots)=0,
\\[4pt]
y_1(y_3+y_4)+q(y_3,y_4)+\xi(y_3,y_4)+y_2\psi_{2}(y_1,1,y_3,y_4)+y_2^2(\cdots)=0.
\end{cases}
\]
Then we get only one new singular point: the origin $O_2$ where 
the singularity of $\tilde H$ is analytically isomorphic to a singularity in 
$\CC^3_{y_1,y_3,y_4}/\muu_{2}(1,1,1)$ given by
\begin{equation}
\label{equation-recent}
\{y_1(y_3+y_4)+q(y_3,y_4)+\text{(terms of degree $\ge 3$)}=0\}.
\end{equation}
Hence, $(\tilde H,O_2)$ is a log terminal singularity of index $2$.
\end{proof}

Therefore, for the graph $\Delta(\hat H, \Gamma+\hat C)$ we have only the following 
two possibilities:
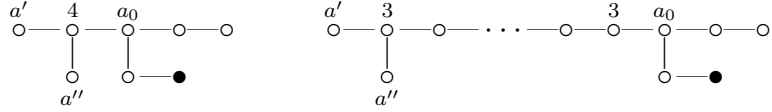
\[
\begin{array}{c@{}c@{}c@{}c@{}c@{}c@{}c@{}c@{}c}
\overset{a'}{\circ}&\lin&\overset{4}{\circ}&\lin&\overset{a_0}{\circ}&\lin&\circ&\lin&\circ
\\[-5pt]
&&|&&|
\\[-5pt]
&&\underset{a''}{\circ}&&\circ&\lin&\bullet
\end{array}
\qquad
\begin{array}{c@{}c@{}c@{}c@{}c@{}c@{}c@{}c@{}c@{}c@{}c@{}c@{}c@{}c@{}c}
\overset{a'}{\circ}&\lin&\overset{3}{\circ}&\lin&\circ&\lin\cdots\lin&\circ&\lin&\overset{3}{\circ}&\lin&\overset{a_0}{\circ}&\lin&\circ&\lin&\circ
\\[-5pt]
&&|&&&&&&&&|
\\[-5pt]
&&\underset{a''}{\circ}&&&&&&&&\circ&\lin&\bullet
\end{array}
\]
where the vertex marked by $a_0$ (resp. $a'$, $a''$)
corresponds to $\Xi_0$ (resp. $\Xi'$, $\Xi''$)
and $\bullet$ corresponds to $\hat C$. 

Using Lemma \ref{IIB-lemma-computations-self-intersections-1},
\ref{IIB-computations-self-intersections-1} 
one can easily obtain that $a_0=2$.
Similarly, 
\begin{equation*}
\label{IIB-computations-self-intersections-2}
(\Xi'\cdot \Xi)_{\tilde H}=(\Xi''\cdot \Xi)_{\tilde H}=-2.
\end{equation*}
This gives us $a'=a''=3$. 
However
the second of the above configurations is not contractible. 
We get the case \ref{IIB-theorem-A2case-simple}.
\end{sde}

\begin{sthm}{\bf Corollary.}\label{lemma-IIB-q4ne0}
$q(0,y_4)\neq 0$.
\end{sthm}
\begin{proof}
Assume that $q(0,y_4)= 0$. 
Take $H$ so that, in \eqref{equation-IIB-H},
functions $\eta$, $\phi$, $l$, $q$, $\xi$, and $\psi$ are sufficiently general 
under this assumption. 
Let $X'$ be a general one-parameter deformation family of $H$.
According to \cite[Prop. 11.4]{Kollar-Mori-1992} 
there is a contraction $f':X'\to Z'$, so $(X',C)$ is an extremal curve germ. 
Moreover, $(X',C')$ is of type \type{IIB}.
By \ref{II-B-case-most-general} we get a contradiction
(otherwise \eqref{equation-recent} is not a point of type $\frac14(1,1)$). 
\end{proof}

\begin{sde}{\bf Subcase: $(X,P)$ is simple \type{cAx/4}-point and $c=0$.}
We shall show that only the case \ref{IIB-theorem-smooth-case-simple} occurs.
Equations \eqref{equation-IIB-rank=2} have the form
\begin{equation*}
\begin{cases}
 y_1^2-y_2^3+ y_3y_4+\phi=0,
\\[4pt]
y_1y_3+y_2q(y_3,y_4)+\xi(y_3,y_4)+\psi =0.
\end{cases}
\end{equation*}
In this case, $\Xi=3\Xi_0+\Xi'+\Xi''$, where $\Xi'$ and $\Xi''$
are given in $E\simeq \PP(3,2,1,1)$ as follows:
\[
\begin{array}{l}
\Xi'=\{y_4=y_1+y_2q(y_3,0)/y_3+\xi(y_3,0)/y_3=0\},
\\[5pt]
\Xi''=\{y_3=y_2q(0,y_4)/y_4^2+\xi(0,y_4)/y_4^2=0\}. 
\end{array}
\]

\begin{ssthm}{\bf Claim.}
The surface $\tilde H$ is normal and has the following singularities
\textup(in natural weighted coordinates on $E\simeq \PP(3,2,1,1)$\textup)\textup:
\begin{itemize}
\item 
$O_1:=\Xi_0\cap \Xi''=(1:0:0:0)$ which is of type \type{A_2},
\item 
$Q:=\Xi_0\cap \tilde C=(1:1:0:0)$ which is of type \type{A_2},
\item 
$O_2:=\Xi_0\cap \Xi'=(0:1:0:0)$ which is of type 
$\frac{1}{4}(1,1)$.
\end{itemize}
The pair $(\tilde H,\Xi_0+\Xi'+\Xi''+\tilde C)$ is LC.
Thus $\tilde H$ looks as follows\textup: 
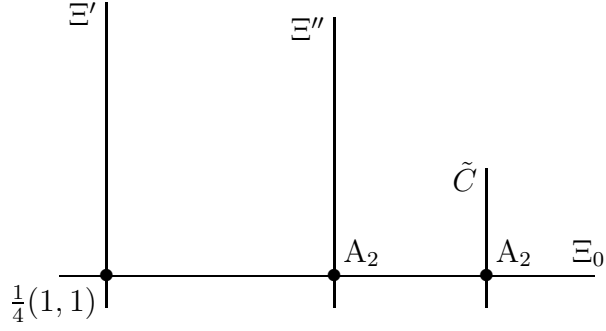
\[
\xy
(55,0)="A" *{}*+!DR{\tilde C},
(55,-10)="B" *{}*+!DR{},
(55,-10)="FF" *{\bullet}*+!DL{\mathrm A_2},
(5,-10)="C" *{}*+!DR{},
(5,22)="D" *{}*+!DR{\Xi'},
(3,-10)="E" *{}*+!DL{},
(65,-10)="F" *{}*+!DL{\Xi_0},
(35,-10)="X" *{\bullet}*+!DL{\mathrm A_2},
(35,20)="Y" *{}*+!DR{\Xi''},
(25,15)="TT" *{}*+!DL{},
{"E";"F":"C";"D",x} ="I" *{\bullet}*+!UR{\frac14(1,1)},
"B";"A"**{} +/1pc/;-/1pc/ **@{-},
"C";"D"**{} +/1pc/;-/1pc/ **@{-},
"E";"F"**{} +/1pc/;-/1pc/ **@{-},
"X";"Y"**{} +/1pc/;-/1pc/ **@{-}
\endxy
\]
\end{ssthm}
The proof is similar to the proof of Claim \ref{IIB-claim-1}, so we omit it.

By the above claim $\Delta(H,C)$ has the form
\[
\begin{array}{c@{}c@{}c@{}c@{}c@{}c@{}c@{}c@{}c@{}c@{}c@{}c@{}c@{}}
\overset{a''}{\circ}&\lin&\circ&\lin&\circ&\lin&\overset{a_0}\circ&\lin&\circ&\lin&\circ&\lin&\bullet 
\\[-3pt]
&&&&&&|
\\[-3pt]
&&&&\underset{a'}\circ&\lin&\underset4\circ
\end{array}
\]
Since 
\begin{equation*}
\label{IIB-computations-self-intersections-3}
(\Xi'\cdot \Xi)_{\tilde H}=-2,\quad 
(\Xi''\cdot \Xi)_{\tilde H}=-\frac43.
\end{equation*}
(cf. Lemma \ref{IIB-lemma-computations-self-intersections-1},
\ref{IIB-computations-self-intersections-1}), we have 
$a_0=2$ and $a'=a''=3$. Thus we get the case \ref{IIB-theorem-smooth-case-simple}.
\end{sde}
\end{de}

\begin{de} 
{\bf Case of double \type{cAx/4}-point.}
We may assume that $\eta=y_3^2$. 
By Corollary \ref{lemma-IIB-q4ne0} $q(0,y_4)\neq 0$, so we also may assume that
$q(0,y_4)=y_4^2$.
Thus the equations \eqref{equation-IIB-H} for $(H,P)$ have the form:
\begin{equation*}
\label{equation-IIB-rank=1}
\begin{cases}
 y_1^2-y_2^3+ y_3^2+\phi=0,
\\[4pt]
y_1l(y_3,\, y_4)+y_2q(y_3,y_4)+\xi(y_3,y_4)+\psi =0.
\end{cases}
\end{equation*}
where $\phi$ does not contain any terms of degree $\le 2$.
This case is more complicated because $\tilde X$ has 
non-isolated singularities:

\begin{sde} {\bf Remark.}
$\Sing(\tilde X)$ has exactly one one-dimensional 
irreducible component 
\[
\Lambda:=\{y_3=y_1^2-y_2^3+\phi_{\sigma=3/2}(y_1,y_2,0,y_4)=0\}\subset E\simeq \PP(3,2,1,1).
\]
\end{sde}

There are the following subcases.

\begin{sde}{\bf Subcase: $(X,P)$ is double \type{cAx/4}-point and $l(0,\, y_4)\neq 0$.}
We shall show that only the case \ref{IIB-theorem-D4case-double} occurs.
After a coordinate change, we may assume that $l(y_3,\, y_4)=y_4$, so 
the equations \eqref{equation-IIB-H} for $(H,P)$ have the form:
\begin{equation}
\label{equation-IIB-H-double}
\begin{cases}
 y_1^2-y_2^3+ y_3^2+\phi=0,
\\[4pt]
y_1y_4+y_2q(y_3,y_4)+\xi(y_3,y_4)+\psi =0.
\end{cases}
\end{equation}
In this case, $\Xi=2\Xi_0+2\Xi'$, where 
\[
\Xi'=\{y_3=y_1+y_2q(0,y_4)/y_4+\xi(0,y_4)/y_4=0\}\subset E\simeq \PP(3,2,1,1).
\] 
\begin{ssthm}{\bf Claim.}
The surface $\tilde H$ is normal and has the following singularities on $\Xi_0$
\textup(in natural weighted coordinates on $E\simeq \PP(3,2,1,1)$\textup)\textup:
\begin{itemize}
\item 
$O_1:=(1:0:0:0)$ which is of type \type{A_2},
\item 
$Q:=\Xi_0\cap \tilde C=(1:1:0:0)$ which is of type \type{A_1},
\item 
$O_2:=\Xi_0\cap \Xi'=(0:1:0:0)$ which is a log terminal point of index $2$.
\end{itemize}
The pair $(\tilde H,\Xi_0+\Xi'+\tilde C)$ is LC along $\Xi_0$.
Moreover, it is PLT at all points of 
$\Xi_0\setminus \{O_2,\, Q\}$. 
Thus $\tilde H$ looks as follows\textup: 
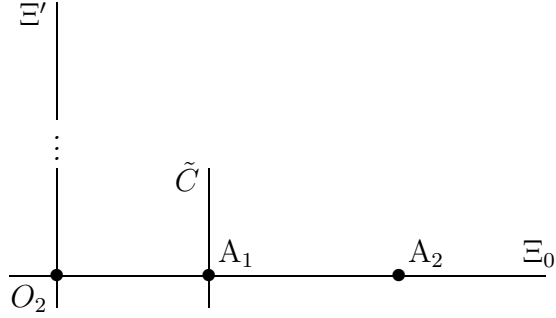
\[
\xy
(25,0)="A" *{}*+!DR{\tilde C},
(25,-10)="B" *{}*+!DR{},
(5,-10)="C" *{}*+!DR{},
(5,22)="D" *{}*+!DR{\Xi'},
(3,-10)="E" *{}*+!DL{},
(65,-10)="F" *{}*+!DL{\Xi_0},
(50,-10)="EE" *{\bullet}*+!DL{\mathrm A_2},
(7,-8)="X" *{}*+!DL{},
(25,15)="TT" *{}*+!DL{},
(25,-10)="FF" *{\bullet}*+!DL{\mathrm A_1},
{"E";"F":"C";"D",x} ="I" *{\bullet}*+!UR{O_2},
"B";"A"**{} +/1pc/;-/1pc/ **@{-},
(5,0)="D1" *{}*+!DR{},
(5,8)="G" *{\vdots}*+!UR{},
(5,15)="D2" *{}*+!DR{},
"C";"D1"**{} +/1pc/;-/1pc/ **@{-},
"D2";"D"**{} +/1pc/;-/1pc/ **@{-},
"E";"F"**{} +/1pc/;-/1pc/ **@{-},
\endxy
\]
where there are more singular points sitting on $\Xi'\setminus \{O_2\}$ 
which must be Du Val. 
\end{ssthm}

\begin{ssde}{\bf Remark.}
For general choice of $\xi$ and $\phi$ the surface $\tilde H$ has 
exactly three singular points 
on $\Xi'\setminus \{O_2\}$ and these points are of type \type{A_1}.
 
\end{ssde}

Hence the dual graph $\Delta(H, C)$ has one of the following forms:
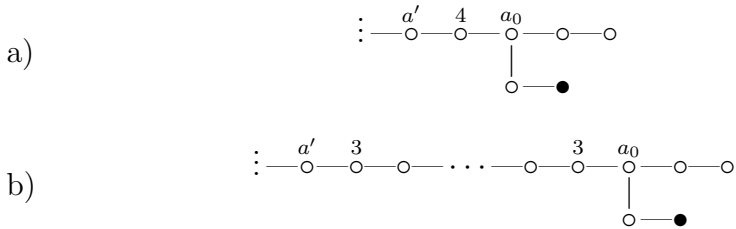
\[
\begin{array}{c@\,c@{}c@{}c@{}c@{}c@{}c@{}c@{}c@{}c@{}c@{}c}
\vdots&\lin&\overset{a'}{\circ}&\lin&\overset{4}{\circ}&\lin&\overset{a_0}\circ&\lin&\circ&\lin&\circ&
\\[-4pt]
&&&&&&|
\\[-4pt]
&&&&&&\circ&\lin&\bullet
\end{array}\leqno{\mathrm a)}
\]
\[
\begin{array}{c@\,c@{}c@{}c@{}c@{}c@{}c@{}c@{}c@{}c@{}c@{}c@{}c@{}c@{}c@{}c@{}c@{}c}
\vdots&\lin&\overset{a'}{\circ}&\lin&
\overset{3}{\circ}&\lin&\circ&\lin\cdots\lin&\circ&\lin&\overset{3}{\circ}
\lin&\overset{a_0}\circ&\lin&\circ&\lin&\circ&
\\[-4pt]
&&&&&&&&&&&|
\\[-4pt]
&&&&&&&&&&&\circ&\lin&\bullet
\end{array}\leqno{\mathrm b)}
\]
where $\vdots$ corresponds to some Du Val singularities sitting on $\Xi'$.
Since the whole configuration is contractible to either a Du Val point or a curve,
we have $a_0=2$ and the case b) does not occur.
In the case a), contracting black vertices successively we get
\[
\begin{array}{c@{\hspace{2pt}}c@{}c}
\vdots&\lin&\overset{a'-1}{\circ}
\end{array}
\]
Hence $a'=2$ or $3$. 

\begin{ssde}{}\label{codiscrepancies-IIB}
Let $(S,o)$ be a normal surface singularity and let $\mu: \hat S\to S$
be its resolution. Recall that the \textit{codiscrepancy divisor } is a unique
$\QQ$-divisor $\Theta=\sum \theta_i\Theta_i$ on $\hat S$ with support in the exceptional locus such that 
$\mu^*K_S=K_{\hat S}+\Theta$. 
If $\mu$ is the minimal resolution, then $\Theta$ must be effective.
The coefficient $\theta_i$ is called the \textit{codiscrepancy} of 
$\Theta_i$. We denote it by $\cd(\Theta_i)$. If $(S,o)$ is a rational singularity, then 
$\theta_i=\cd(\Theta_i)$ can be found from the following system of linear equations:
\[
\sum_i \theta_i\Theta_i\cdot \Theta_j=-K_{\hat S}\cdot \Theta_j=2+\Theta_j^2.
\]
Let $a_i:=-\Theta_i^2$. Then the system can be rewritten as follows:
\[
a_j\theta_j= -\Theta_j^2-2+\sum\nolimits' \theta_i
\]
where $\sum'$ runs through all exceptional curves $\Theta_i$ meeting $\Theta_j$.
\end{ssde}
\begin{ssthm}{\bf Corollary.}\label{discrepancies-chain-Dn}
Let $\Delta$ be the dual graph of a resolution of a rational singularity
and let $\Delta'$ be its subgraph consisting of one vertex of weight $a\ge 2$
and $n-1$ vertices of weight $2$. Assume that the remaining part $\Delta\setminus \Delta'$ 
is attached to $\overset a\circ$.
\begin{enumerate}
 \item 
If $\Delta'$ has the form
\[
\begin{array}{c@{}c@\,c@\,c@{}c@{}c}
\circ&\lin&\cdots&\lin&\circ&\lin\overset a\circ\cdots 
\end{array} 
\]
then the codiscrepancies of the corresponding to $\Delta'$ exceptional components,
indexed from the left to right, 
are computed by $\alpha_k=k\alpha_1$, $k\le n$. 
 \item 
If $\Delta'$ has the form
\[
\begin{array}{c@{}c@{}c@{}c@\,c@\,c@{}c@{}c}
\circ&\lin&\circ&\lin&\cdots&\lin&\circ&\lin\overset a\circ\cdots 
\\[-5pt]
&&|
\\[-5pt]
&&\circ
\end{array} 
\]
then the codiscrepancies of the corresponding to $\Delta'$ exceptional components
are computed by $\alpha_1=\alpha_2=2\alpha_3$ and $\alpha_k=\alpha_3$ for $3\le k\le n$. 
\end{enumerate}
\end{ssthm}
\begin{ssde}
By Lemma \ref{IIB-lemma-computations-self-intersections-1},
\ref{equation-IIB-discrepansies} we have
$\cd(\Xi_0)=\cd(\Xi')=3/2$. 
Using \ref{codiscrepancies-IIB} we compute the codiscrepancies of exceptional divisors 
over $\tilde H$:
\[
\begin{array}{c@\,c@{}c@{}c@{}c@{}c@{}c@{}c@{}c@{}c@{}c@{}c}
\vdots&\lin&\overset{3/2}{\circ}&\lin&\overset{5/4}{\circ}&\lin&
\overset{3/2}\circ&\lin&\overset1\circ&\lin&\overset{1/2}\circ&
\\[-3pt]
&&&&&&|
\\[-3pt]
&&&&&&\underset{3/4}{\circ}&\lin&\bullet
\end{array}
\]
\end{ssde}
\begin{ssde}{}\label{IIB-double-subcase-a_22}
If $a'=2$, then the configuration $\vdots\lin\overset{a'-1}{\circ}$ is contracted either to 
a smooth point or to a curve. Therefore we have one of the following possibilities:
\[
\begin{array}{c@\,c@\,c@\,c@\,c@\,c@\,c@\,c@\,c@\,c@\,c@\,c@\,c@\,c@\,c@\,c}
\overset{\alpha_1}{\circ}&\lin&\cdots&\lin&\overset{\alpha_n}{\circ}&\lin&
\overset{3/2}{\circ}&\lin&\overset{5/4}{\circ}&\lin&
\overset{3/2}{\circ}&\lin&\overset1{\circ}&\lin&\overset{1/2}{\circ}&
\\[-3pt]
&&&&&&&&&&|
\\[-3pt]
&&&&&&&&&&\underset{3/4}{\circ}&\lin&\bullet
\end{array}\leqno{\mathrm a1)}
\]

\[
\begin{array}{c@\,c@\,c@\,c@\,c@\,c@\,c@\,c@\,c@\,c@\,c@\,c@\,c@\,c@\,c@\,c@\,c@\,c}
\overset{\alpha_1}{\circ}&\lin&\overset{\alpha_3}{\circ}&\lin&\cdots&\lin&\overset{\alpha_n}{\circ}&\lin&
\overset{3/2}{\circ}&\lin&\overset{5/4}{\circ}&\lin&
\overset{3/2}{\circ}&\lin&\overset1{\circ}&\lin&\overset{1/2}{\circ}&
\\[-3pt]
&&|&&&&&&&&&&|
\\[-3pt]
&&\underset{\alpha_2}\circ&&&&&&&&&&\underset{3/4}{\circ}&\lin&\bullet
\end{array}\qquad n\ge 2\leqno{\mathrm a2)}
\]
Then we get a contradiction by Corollary \ref{discrepancies-chain-Dn}.
\end{ssde}

\begin{ssde}{}\label{IIB-double-subcase-a_23}
Thus, $a'=3$. 
Then $f$ is divisorial and the configuration $\vdots\lin\overset{a'-1}{\circ}$ 
is exactly the dual graph of the minimal resolution of $(T,o)$ which is a Du Val 
graph of type \type{E_6},
\type{D_5}, \type{D_4}, \type{A_4}, \type{A_3}, \type{A_2} or \type{A_1}.
If the graph $\Delta(H,C)$ has the form a1), then, 
as above, $3/2=\alpha_{n+1}=(n+1)\alpha_1$, $3\cdot 3/2=1+\alpha_n+5/4$.
This gives us $n\alpha_1=9/4$, $\alpha_1=3/2-9/4<0$, a contradiction.
Similarly, in the case a2) with $n\ge 3$ we obtain $\alpha_n=3/2$, $3\cdot 3/2=1+\alpha_n+5/4$,
a contradiction.

If there are three connected components of the exceptional divisor attached to $\Xi'$,
then for corresponding codiscrepancies $\alpha_n$, $\beta_m$, $\gamma_l$ we have
$3\cdot 3/2=1+\alpha_n+\beta_m+\gamma_l+5/4$, $\alpha_n+\beta_m+\gamma_l=9/4$.
On the other hand, $2\alpha_n\ge 3/2$, $2\beta_m\ge 3/2$, $2\gamma_l\ge 3/2$.
Hence the equalities $\alpha_n=\beta_m=\gamma_l=3/4$ hold and we 
the get case \ref{IIB-theorem-D4case-double}.

In the remaining cases, by direct computations we obtain that the exceptional divisors 
have codiscrepancies whose denominators divide $4$ 
only in cases \ref{IIB-theorem-D5case-double} or \ref{IIB-theorem-E6case-double} below.
\end{ssde}

\begin{ssde}\label{IIB-theorem-D5case-double}
$(T,o)$ is Du Val of type \type{D_5} and $\Delta(H,C)$ has the form
\[
\begin{array}{c@\,c@\,c@\,c@\,c@\,c@\,c@\,c@\, c@\, c@\,c@\,c@\,c}
\circ&\lin&\circ&\lin&\overset{\Xi_0}{\circ}&
\lin&\overset4\circ&\lin&\overset3{\circ}&\lin&\circ&\lin&\circ
\\[-5pt]
&&&&|&&&&|&&|
\\[-5pt]
&&\bullet&\lin&\circ&&&&\underset{}{\circ}&&\circ
\end{array}
\]
here $\tilde H$ has two singular points on $\Xi'\setminus \Xi_0$ and these points are 
of types \type{A_1} and \type{A_3}.
\end{ssde}

\begin{ssde}\label{IIB-theorem-E6case-double}
$(T,o)$ is Du Val of type \type{E_6} and $\Delta(H,C)$ has the form
\[
\begin{array}{c@\,c@\,c@\,c@\,c@\,c@\,c@\,c@\, c@\, c@\,c@\,c@\,c@\,c@\,c@\,}
\circ&\lin&\circ&\lin&\overset{\Xi_0}{\circ}&\lin&\overset4\circ&\lin&\overset3{\underset{\Xi'}{\circ}}
&\lin&\circ&\lin&\circ&\lin&\circ
\\[-9pt]
&&&&|&&&&&&|
\\[-5pt]
&&\bullet&\lin&\circ&&&&&&\circ&\lin&\circ
\end{array}
\]
here $\tilde H$ has exactly one singular points on $\Xi'\setminus \Xi_0$ and this point is
of type \type{A_5}.
\end{ssde}
\end{sde}

\begin{sde}\label{general-extension}
Now we show that in cases \ref{IIB-theorem-D5case-double} and \ref{IIB-theorem-E6case-double}
the chosen element $H\in |\OOO_X|_C$ is not general. 
Consider the case \ref{IIB-theorem-D5case-double}. Case \ref{IIB-theorem-E6case-double}
can be treated similarly. Take a divisor $D$ on $\hat H$ whose coefficients are 
as follows:
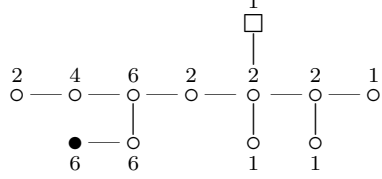
\[
\begin{array}{c@\,c@\,c@\,c@\,c@\,c@\,c@\,c@\, c@\, c@\,c@\,c@\,c}
&&&&&&&&\overset 1\Box
\\[-4pt]
&&&&&&&&|
\\[-1pt]
\overset2\circ&\lin&\overset4\circ&\lin&\overset6\circ&\lin&\overset2\circ&\lin&
\overset2\circ&\lin&\overset2\circ&\lin&\overset1\circ
\\[-5pt]
&&&&|&&&&|&&|
\\[-5pt]
&&\underset6\bullet&\lin&\underset6\circ&&&&\underset{1}{\circ}&&\underset1\circ
\end{array}
\]
where $\Box$ corresponds to 
an arbitrary smooth analytic curve $\hat G$ meeting $\Xi'$ transversely
so that $\Supp D$ is a simple normal crossing divisor.
It is easy to verify that $D$ is numerically trivial, so $D=h^* G_Z$, where $G_Z$ is a
Cartier divisor on $T$.
Since $R^1f_*\OOO_X=0$, by the exact sequence
\[
0 \longrightarrow \OOO_{X} \longrightarrow\OOO_{X}(H')\longrightarrow
\OOO_H(H')\longrightarrow 0
\]
we get surjectivity of the map $H^0(X,\OOO_{X}(H'))\to H^0(H,\OOO_{H}(H'))$.
Thus there is a member $H'\in |\OOO_X|_C$ such that $H'|_H=f_H^*G_Z$.

The proper transform $\tilde H'$ of $H'$ by $g$ 
satisfies $\tilde H'=g^*H' -E|_{\tilde X}$.
Since $\Xi=E \cap \tilde{H}$ and $\Xi=2 \Xi_0 + 2 \Xi'$, 
we have $\tilde H'|_{\tilde{H}} = 4 \Xi_0 +g_1(\hat G)$.
In particular, $\Xi'$ is not a component of $\tilde H'|_{\tilde{H}}$.
Note that $|g_1(\hat G)|$ is a base point free linear system on $\tilde{H}$
(because $H^1(\OOO_{\tilde{H}})=0$). 
Thus we can take $H'$ so that $\tilde H'$ does not pass 
through points in $\tilde H\cap \Lambda\setminus \Xi_0$.
Now let $H_\epsilon$ be a general member of the pencil generated by 
$H$ and $H'$. 
Note that $\Lambda\cap \Xi_0=\{Q\}$
and $\Lambda$ meets $\tilde H$ and $\tilde H_\epsilon$ transversely at $Q$.
By Bertini's theorem the proper transform $\tilde H_\epsilon$ of $H_\epsilon$
on $\tilde X$ meets $\Lambda$ transversely also along $\Xi'$. 
Since $(\tilde H_\epsilon\cdot \Lambda)_{\tilde X}= (\OOO(4)\cdot \Lambda)_{\PP(3,2,1,1)}=4$,
the intersection $\tilde H_\epsilon\cap \Lambda$ consists of four 
distinct points. Therefore, $\tilde H_\epsilon$ has three Du Val points on 
$\tilde H_\epsilon\cap \Lambda\setminus \Xi_0$. This shows that 
for $H_\epsilon$ the situation of \ref{IIB-theorem-D4case-double} holds, so the chosen $H$ is not general
in the case \ref{IIB-theorem-D5case-double}.
\end{sde}

\begin{sde}{\bf Subcase: $(X,P)$ is double \type{cAx/4}-point and $l(0,\, y_4)= 0$.}
We shall show that only the case \ref{IIB-theorem-conic-bundle-case-double} occurs.
We may assume that $l(y_3,\, y_4)=y_3$, so the equations 
\eqref{equation-IIB-H} for $(H,P)$ have the form:
\begin{equation}
\label{equation-IIB-H-double-0}
\begin{cases}
 y_1^2-y_2^3+ y_3^2+\phi=0,
\\[4pt]
y_1y_3+y_2q(y_3,y_4)+\xi(y_3,y_4)+\psi =0.
\end{cases}
\end{equation}
In this case, $\Xi=4\Xi_0+2\Xi'$, where 
\[
\Xi'=\{y_3=y_2q(0,y_4)/y_4^2+\xi(0,y_4)/y_4^2=0\}\subset E\simeq \PP(3,2,1,1).
\] 

\begin{ssthm}{\bf Claim.}
The surface $\tilde H$ is normal and has the following singularities on $\Xi_0$
\textup(in natural weighted coordinates on $E\simeq \PP(3,2,1,1)$\textup)\textup:
\begin{itemize}
\item 
$O_1:=\Xi_0\cap \Xi'=(1:0:0:0)$ which is of type \type{A_2},
\item 
$Q:=\Xi_0\cap \tilde C=(1:1:0:0)$ which is of type \type{A_3},
\item 
$O_2:=(0:1:0:0)$ which is a cyclic quotient singularity of type $\frac14(1,1)$.
\end{itemize}
The pair $(\tilde H,\Xi_0+\Xi'+\tilde C)$ is LC along $\Xi_0$.
Moreover, it is PLT at all points of 
$\Xi_0\setminus \{O_1,\, Q\}$. 
Thus $\tilde H$ looks as follows\textup: 
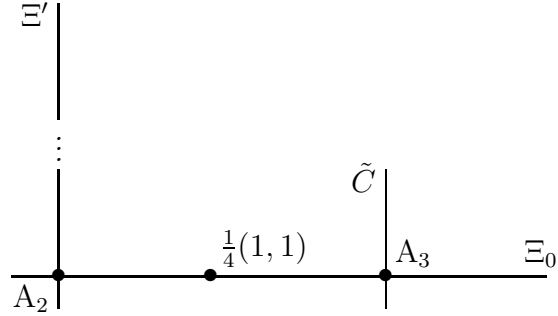
\[
\xy
(48,0)="A" *{}*+!DR{\tilde C},
(48,-10)="B" *{}*+!DR{},
(5,-10)="C" *{}*+!DR{},
(5,22)="D" *{}*+!DR{\Xi'},
(3,-10)="E" *{}*+!DL{},
(65,-10)="F" *{}*+!DL{\Xi_0},
(25,-10)="EE" *{\bullet}*+!DL{\frac14(1,1)},
(7,-8)="X" *{}*+!DL{},
(25,15)="TT" *{}*+!DL{},
(48,-10)="FF" *{\bullet}*+!DL{\mathrm A_3},
{"E";"F":"C";"D",x} ="I" *{\bullet}*+!UR{\mathrm A_2},
(5,0)="D1" *{}*+!DR{},
(5,8)="G" *{\vdots}*+!UR{},
(5,15)="D2" *{}*+!DR{},
"B";"A"**{} +/1pc/;-/1pc/ **@{-},
"C";"D1"**{} +/1pc/;-/1pc/ **@{-},
"D2";"D"**{} +/1pc/;-/1pc/ **@{-},
"E";"F"**{} +/1pc/;-/1pc/ **@{-},
\endxy
\]
\end{ssthm}

Hence the dual graph $\Delta(H, C)$ has the following form:
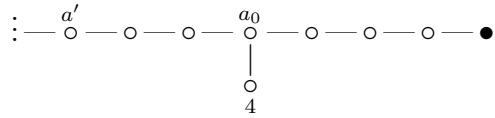
\[
\begin{array}{c@\,c@\,c@\,c@\,c@\,c@\,c@\,c@\,c@\,c@\,c@\,c@\,c@\,c@\,c@\,c@\,c@\,}
\vdots&\lin&\overset{a'}{\circ}&\lin&\circ &\lin&\circ&\lin&\overset{a_0}{\circ}
&\lin&\circ&\lin&\circ&\lin&\circ&\lin&\bullet
\\[-4pt]
&&&&&&&&|
\\[-4pt]
&&&&&&&&\underset{4}{\circ}
\end{array}
\]
where $\vdots$ corresponds to some Du Val singularities sitting on $\Xi'$.
Since the whole configuration is contractible to either a Du Val point or a curve,
we have $a_0=2$. Contracting black vertices successively on some step we get
\[
\begin{array}{c@\,c@\,c}
\vdots&\lin&\overset{a'-2}{\circ}
\end{array}
\]
Recall that $\vdots$ is not empty.
Hence $a'=3$ or $4$. By Lemma \ref{IIB-lemma-computations-self-intersections-1}, 
\ref{equation-IIB-discrepansies} we have
$\cd(\Xi_0)=3$, $\cd(\Xi')=3/2$. 
Using \ref{codiscrepancies-IIB} we compute the codiscrepancies of exceptional divisors 
over $\tilde H$:

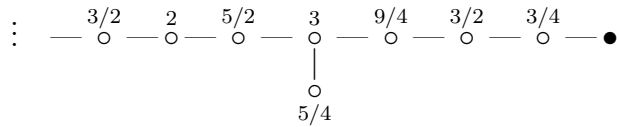
\[
\begin{array}{c@{\quad }c@\,c@\,c@\,c@\,c@\,c@\,c@\,c@\,c@\,c@\,c@\,c@\,c@\,c@\,c@\,c@\,}
\vdots&\lin&\overset{3/2}{\circ}&\lin&\overset{2}\circ &\lin&\overset{5/2}\circ&\lin&\overset{3}{\circ}
&\lin&\overset{9/4}\circ&\lin&\overset{3/2}\circ&\lin&\overset{3/4}{\circ}&\lin&\bullet
\\[-4pt]
&&&&&&&&|
\\[-4pt]
&&&&&&&&\underset{5/4}{\circ}
\end{array}
\]
If $a'=4$, we get a contradiction as in \ref{IIB-double-subcase-a_23}.
If $a'=3$, then the whole configuration contracts to a curve, i.e.,
$f$ is a $\QQ$-conic bundle.
As in \ref{IIB-double-subcase-a_22} we infer that
the graph $\Delta(H,C)$ has the following form
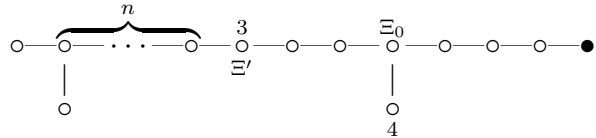
\[
\begin{array}{c@{}c@{}l@{}c@{}c@{}c@{}c@{}c@{}c@{}c@{}c@{}c@{}c@{}c@{}c@{}c@{}c@{}c@{}c}
\circ&\lin&\overbrace{\circ\lin\cdots\lin\circ}^n&\lin&\overset{3}{\underset{\Xi'}{\circ}}&\lin
&\overset{}\circ &\lin&\overset{}\circ&\lin&\overset{\Xi_0}{\circ}
&\lin&\overset{}\circ&\lin&\overset{}\circ&\lin&\overset{}{\circ}&\lin&\bullet
\\[-6pt]
&&\hspace{1.3pt}|&&&&&&&&|
\\[-3pt]
&&\circ&&&&&&&&\underset{4}{\circ}
\end{array}
\]
where $n\ge 0$. 

We show that $n=0$, that is,
the case \ref{IIB-theorem-conic-bundle-case-double} holds.
As in \ref{general-extension} take a divisor $D$ on $\hat H$ 
whose coefficients are as follows
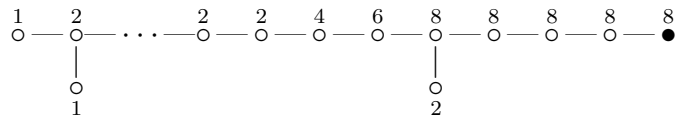
\[
\begin{array}{c@\,c@\,l@\,c@\,c@\,c@\,c@\,c@\,c@\,c@\,c@\,c@\,c@\,c@\,c@\,c@\,c@\,c@\,c@\,}
\overset{1}\circ&\lin&\overset{2}\circ\lin\cdots\lin\overset{2}\circ&\lin&\overset{2}{\circ}&\lin&\overset{4}\circ 
&\lin&\overset{6}\circ&\lin&\overset{8}{\circ}
&\lin&\overset{8}\circ&\lin&\overset{8}\circ&\lin&\overset{8}{\circ}&\lin&\overset 8\bullet
\\[-4pt]
&&\hspace{1.3pt}|&&&&&&&&|
\\[-4pt]
&&\underset{1}\circ&&&&&&&&\underset{2}{\circ}
\end{array}
\]
Then $D=h^*o$ is a scheme fiber of $h: \hat H\to T$.
There is a member $H'\in |\OOO_X|_C$ such that $H'|_H=g_{H*}g_{1*}D=
f_H^*o$. 
Since $\Xi=4\Xi_0+2\Xi'$, we have $\tilde H'|_{\tilde H}=
g_{1*}D-\Xi=4\Xi_0$. In particular, 
the curve $\Xi'$ is not a component of $\tilde H'|_{\tilde H}$. 
Hence the base locus of the pencil generated by $\tilde H$ and $\tilde H'$
coincides with $\Xi_0$. As in \ref{general-extension} 
a general member $\tilde H_\epsilon$ of this pencil 
meets the curve $\Lambda$ transversely outside of $\Xi_0$.
Note that $\Lambda\cap \Xi_0=\{Q\}$
and the local intersection number of $\Lambda$ and $\tilde H_\epsilon$ at $Q$
equals to $2$.
By Bertini's theorem the proper transform $\tilde H_\epsilon$ of $H_\epsilon$
on $\tilde X$ meets $\Lambda$ transversely along $\Xi'$. 
Since $(\tilde H_\epsilon\cdot \Lambda)_{\tilde X}= (\OOO(4)\cdot \Lambda)_{\PP(3,2,1,1)}=4$,
the intersection $\tilde H_\epsilon\cap \Lambda$ consists of three 
distinct points. Therefore, $\tilde H_\epsilon$ has two Du Val points on 
$\tilde H_\epsilon\cap \Lambda\setminus \Xi_0$. This shows that 
for $H_\epsilon$ the situation of \ref{IIB-theorem-conic-bundle-case-double} holds, 
so the chosen $H$ is not general if $n>0$.

\begin{ssde}{\bf Example.}
Let $H$ be given by the equations
\begin{equation*}
\begin{cases}
 y_1^2-y_2^3+ y_3^2=0,
\\[4pt]
y_1y_3+y_2y_4^2+y_4^{4} =0.
\end{cases}
\end{equation*}
Then a one-parameter deformation of $H$ 
is a $\QQ$-conic bundle as in \ref{IIB-theorem-conic-bundle-case-double}.
\end{ssde}
\end{sde}
\end{de}

\par\medskip\noindent
{\bf Acknowledgments.}
The paper was written during the second author's stay at RIMS, Kyoto University
in February-March 2011. 
The author is very grateful to the institute for 
the invitation, hospitality and nice working environment.



\end{document}